\documentclass[notitlepage,11pt,reqno]{amsart}
\usepackage{color}
\usepackage{amssymb,nicefrac,bm,upgreek,mathtools,verbatim,enumerate}
\usepackage[mathscr]{eucal}
\usepackage{graphicx}
\usepackage{epsf,times}

\definecolor{dmagenta}{rgb}{.4,.1,.5}
\definecolor{007}{rgb}{.0,.0,.7}
\definecolor{dred}{rgb}{.5,.0,.0}
\definecolor{dgreen}{rgb}{.0,.5,.0}
\definecolor{dblue}{rgb}{.0,.0,.5}

\definecolor{violet}{rgb}{.3,.0,.9}
\definecolor{orange}{cmyk}{0,.5,.1,.0}
\definecolor{dcyan}{cmyk}{.5,.0,.0,.0}
\definecolor{dyellow}{cmyk}{.0,.0,.5,.0}
\definecolor{cm}{cmyk}{1,.0,.0,.0}
\numberwithin{equation}{section}
\allowdisplaybreaks

\newtheorem{theorem}{Theorem}[section]
\newtheorem{lemma}{Lemma}[section]
\newtheorem{proposition}{Proposition}[section]
\newtheorem{corollary}{Corollary}[section]

\theoremstyle{definition}
\newtheorem{definition}{Definition}[section]
\newtheorem{example}{Example}[section]

\theoremstyle{remark}
\newtheorem{remark}{Remark}[section]

\newcommand{\cA}{\mathcal{A}}

\newcommand{\grad}{\nabla}

\newcommand{\Lyap}{\mathcal{V}}

\newcommand{\R}{\mathbb{R}}
\newcommand{\bS}{\mathbb{S}}

\newcommand{\Rd}{\mathbb{R}^d}

\newcommand{\sB}{\mathscr{B}}
\newcommand{\cC}{\mathcal{C}}

\newcommand{\sL}{\mathscr{L}}

\newcommand{\Ind}{\mathbb{I}}

\newcommand{\bvnorm}[1]{[\kern-0.45ex[\kern0.1ex #1 \kern0.1ex]\kern-0.45ex]}

\newcommand{\abs}[1]{\lvert#1\rvert}
\newcommand{\norm}[1]{\lVert#1\rVert}

\newcommand{\order}{{\mathscr{O}}} 

\newcommand{\df}{:=}
\newcommand{\infdel}{\Delta_\infty}
\newcommand{\ginfdel}{\Delta^\gamma_\infty}

\DeclareMathOperator{\trace}{trace}
\DeclareMathOperator{\dist}{dist}

\begin{document}

\title[Liouville theorems for infinity Laplacian]
{Liouville theorems for infinity Laplacian  with  gradient and KPP type equation}

\author{Anup Biswas}
\address{Indian Institute of Science Education and Research, Dr.\ Homi Bhabha Road, Pashan, Pune 411008}
\email{anup@iiserpune.ac.in}

\author{Hoang-Hung Vo$^{*}$}
\address{Faculty  of  Mathematics  and  Applications,  Saigon  University,  273  An  Duong  Vuong  st.,  Ward  3,Dist.5, Ho Chi Minh City, Viet Nam}
\email{vhhung@sgu.edu.vn}
\thanks{$^*$ Corresponding author}

\date{}

\begin{abstract}
In this paper, we prove  new Liouville type results for a nonlinear equation involving infinity Laplacian with gradient of the form 
$$\ginfdel u + q(x)\cdot \grad{u} |\grad{u}|^{2-\gamma} + f(x, u)\,=\,0\quad \text{in}\; \Rd,$$
where $\gamma\in [0, 2]$ and $\ginfdel$ is a $(3-\gamma)$-homogeneous operator associated with the infinity Laplacian.  Under the assumptions $\liminf_{|x|\to\infty}\lim_{s\to0}f(x,s)/s^{3-\gamma}>0$  and $q$ is a continuous function vanishing at infinity, we construct a positive bounded solution to the equation and if $f(x,s)/s^{3-\gamma}$ decreasing in $s$, 
we further obtain the uniqueness by improving sliding method for infinity Laplacian operator with  nonlinear gradient. Otherwise, if $\limsup_{|x|\to\infty}\sup_{[\delta_1,\delta_2]}f(x,s)<0$, then nonexistence result holds  provided additionally some suitable conditions.  To this aim, we develop  novel techniques to overcome the difficulties stemming from the degeneracy of infinity Laplacian and nonlinearity of the gradient term. Our approach is based on a new regularity result, the strong maximum principle, and Hopf's lemma for infinity Laplacian involving gradient and potential. We also construct  some examples to illustrate our results. We further investigate some deeper qualitative properties of the principal eigenvalue    of the   corresponding nonlinear operator
$$\ginfdel u + q(x)\cdot \grad{u} |\grad{u}|^{2-\gamma}  + c(x)u^{3-\gamma},$$
with Dirichlet  boundary condition in  smooth bounded domains, which may be  of independent interest. The results obtained here
could be considered as sharp extension of the Liouville type results obtained in \cite{ALT,ASS11,BHR, CEG,SP,S1}. 
\end{abstract}

\maketitle



\section{Introduction}
Infinity Laplacian was first introduced in the pioneering works of G. Aronsson \cite{AG1,AG2,AG3} in the 1960s and this operator appeared while studying 
\textit{absolute minimizer} in a domain of  $\Rd$. Later, infinity Laplacian also found its application in image processing \cite{CMS}. In the seminal work \cite{RJ93}, R. Jensen employed the theory of viscosity solutions of elliptic equations \cite{CIL} to establish the equivalence of \textit{absolute minimal Lipschitz extensions} (AMLE) and  viscosity solutions of the  infinity Laplace equation and then proved the uniqueness of  AMLE for
the first time. Since then, 
it turned out that the theory of viscosity solution is an appropriate instrument for
the study of infinity Laplacian. Equations involving infinity Laplacian  have thus received a lot of attention in the community and became a subject of intensive research in the theory of partial differential equations. In the elegant survey \cite{ACJ},  Aronsson, Crandall and Juutinen  gave a complete and self-contained
exposition to the theory of  AMLE (see also, \cite{C08}). In the celebrated work,
 by using probabilistic methods, Peres,  Schramm, Sheffield, and Wilson \cite{PPSW}  showed that the infinity Laplacian also appear in the tug-of-war games,  where two players try to move
a token in an open set $\mathcal{O}$ toward a favorable spot on the boundary $\partial \mathcal{O}$ corresponding
to a given payoff function $g$ on $\partial \mathcal{O}$. In the developing progress, we are attracted by the nice works \cite{ALT,AS10,BDM,BM12,ES,PJ07,JL05,JLM,SP,PV12, PV13,LW2008,LW2008a} and those also motivate us to the current study of the Liouville type result as aforementioned.

Throughout the paper, given $\gamma\in [0, 2]$, we define the operator $\sL$ as follows
$$\sL u = \ginfdel u + q(x)\cdot \grad u \abs{\grad u}^{2-\gamma}=  \frac{1}{\abs{\grad u}^\gamma} \sum_{i, j=1}^d \partial_{x_i} u\, \partial_{x_i x_j} u\, \partial_{x_j} u + q(x)\cdot \grad u \abs{\grad u}^{2-\gamma}.$$
Note that  $\ginfdel u$ becomes the classical infinity Laplacian for $\gamma=0$ while it is normalized infinity Laplacian for $\gamma=2$. We also denote $\Delta_\infty^0 u$ by $\Delta_\infty u$ for simplicity. In the present work, we are interested in the study of the Liouville
type result, that is, the existence and nonexistence of positive solutions to the equation 
\begin{equation}\label{E1.1}
\sL u + f(x, u)=0\quad \text{in}\; \Rd\,,
\end{equation} 
with several types of the nonlinearity $f$ including identical zero. It should be noted that this operator is  of neither variational nor  divergence forms (exception of the case $\gamma=2$ in the two dimensional space \cite{DGIMR}).  When $\gamma=2$ and $q$ Lipschitz continuous, the operator
$\sL$ appears in certain tug-of-war games \cite[Theorem~1.3]{LNR13}.
The main goal of this paper
is to extend the researches in \cite{ALT,ASS11,SP,S1} to a more general equation with a nonlinear
 gradient and reaction term, for which the techniques used in the mentioned works cannot be applied in this framework.

As is known, the Liouville type result is one of the central topics in the field of partial differential equations because it is not only important itself in understanding many natural phenomena such as the spreading, vanishing and transition (see Berestycki et al. \cite{BHN,BH}) but also related to the theory of regularity \cite{DDF,DDWW,M1,S1}. It is worth mentioning that the best known regularity results
till date are $\cC^{1,\alpha}$ regularity, with $0<\alpha\ll1$, for infinity harmonic functions in the plane due to Evans and Savin \cite{ES} and everywhere differentiability in dimensions $d\geq3$ due to Evans and Smart \cite{ES1} . Later Lindgren \cite{Lindgren} extended the result of  \cite{ES1} to the inhomogeneous case.
Moreover, some sharp Sobolev regularities of $|\nabla u|^\alpha$ has been recently obtained for homogeneous and inhomogeneous infinity Laplace equation by Koch, Zhang and Zhou
in their interesting works \cite{KZZ19a,KZZ19b}.
 As a direct application of regularity estimate, a  Liouville type result for infinity 
 harmonic functions was obtained by Savin  \cite[Theorem 4]{S1}. More precisely, he proved that any  infinity harmonic function growing at most linearly at $\infty$, that is,
$$|u(x)|\leq C(1+|x|)\quad\quad\textrm{for some positive constant $C$},$$
must be linear. Another Liouville type result for infinity Laplacian equation with strong absorptions has been recently obtained by Ara\'{u}jo, Leit\~{a}o and Teixeira \cite{ALT}. Their result asserts that any non-negative viscosity solution to
$$\Delta_\infty u\,=\,\lambda (u_+)^\beta\quad\quad\textrm{for given $\lambda>0,\beta\in[0,3),$}$$
which satisfies the growth condition $$u(x)=o(|x|^{\frac{4}{3-\beta}}),\quad \text{as}\; |x|\to\infty,$$
is necessarily  constant. More precisely, if
$$\limsup_{|x|\to\infty}\frac{u(x)}{|x|^{\frac{4}{3-\beta}}}<\left(\frac{\lambda(3-\beta)^4}{64(1+\beta)}\right)^{\frac{1}{3-\beta}},$$
then $u\equiv0$. Though the elliptic equation involving infinity Laplacian has been extensively investigated in the recent years, there has been limited work on the equation involving infinity Laplacian and gradient. We mention that some closely
related works to the current problem are done by Armstrong, Smart and Somersille \cite{ASS11},
L\'{o}pez-Soriano, Navarro-Climent and Rossi \cite{LNR13},
 Patrizi \cite{SP}, and Birindelli,  Galise, and  Ishii \cite{BGI}.
 Note that in \cite{BGI}, the authors also proved some existence and nonexistence of viscosity solution  for elliptic equation with truncated  Laplacian and general inhomogeneous term in any strictly convex domain, which may also be called the Liouville type result for degenerate equation. The theory of inhomogeneous infinity Laplacian equations is more recent and delicate. In particular, Lu and Wang \cite{LW2008,LW2008a} have first used Perron's method and the standard viscosity solution techniques to establish both the existence and uniqueness of solution to inhomogeneous infinity Laplace equation of the form 
\begin{equation*}
\Delta_\infty^\gamma u(x)=f(x) \quad\quad x\in\mathcal{O},
\end{equation*}
where $\mathcal{O}\subset\R^d$ is a bounded domain and $\gamma\in\{0,2\}$, with Dirichlet boundary condition, provided $f$ has a constant sign. It is also interesting in the works \cite{LW2008,LW2008a} that the uniqueness may fail when $f$ is allowed to change sign. It is worth mentioning that the evolution equations of homogeneous equations involving infinity Laplacian and  porous medium have been well investigated in the elegant works of Portilheiro and V\'{a}zquez \cite{PV12,PV13}. Especially in the context of porous medium, the authors of \cite{PV12} can transform the original equation to 
$$u_t=(m-1)u\frac{1}{|\grad u|^2}\Delta_\infty u+|\grad u|^2\quad\quad x\in\mathcal{O}\,,$$
where $\mathcal{O}\subset\R^d$ is a bounded domain and  $m>1$ is the order of porous medium. Lastly, we would like to mention the recent interesting work of Li, Nguyen and Wang \cite{LNW}, who successfully used  comparison principle (for viscosity solutions)  to derive estimates, symmetry properties and Liouville result for solutions to the class of equations (both degenerate and non-degenerate elliptic  included fully nonlinear Yamabe problem)  in conformal geometry. Another interesting work
that also considered gradient
term with infinity Laplacian operators is \cite{JPR16}. Symmetry and overdetermined problems for infinity Laplacian operators are considered in the important works \cite{BK,CF,CF1}.

Our first contribution in this article comes from our new Liouville type results and the uniqueness for the non-negative viscosity solution of the equation (\ref{E1.1})  for a general $f$, not depending on $u$.
This should be compared with \cite{LW2008,LW2008a} where uniqueness is established
for $f>0$.  Some of the key tools in our 
analysis  are the new regularity result in Lemma \ref{L2.1} and comparison principle Theorem \ref{T2.1}. 
Such results were first
considered by Crandall, Evans and Gariepy \cite{CEG} and later improved by Armstrong,  Smart,  Somersille \cite{ASS11} for equations involving gradient and by Mitake and Tran \cite{MT} for weakly coupled systems. A  strong maximum principle and Hopf's lemma, Theorem \ref{T2.2}, is proved to support the positivity of solution while its existence, Theorem~\ref{T2.3}, holds without any sign-assumption on  $f$. On the other hand,  as a direct consequence of regularity result   and comparison principle, we establish
three new Liouville type results (Theorems~\ref{T2.4}-\ref{T2.6} below). Recall that the
first Liouville property of infinity Laplacian is obtained by Crandall, Evans and Gariepy \cite{CEG}, 
which shows that any  supersolutions
$u$ of $-\infdel u=0$ in $\Rd$, which are bounded below are necessarily  constant. We extend this result
in Theorem~\ref{T2.4} by proving
that any locally Lipschitz supersolution $u$, which are bounded  below,  to 
$$-\ginfdel u + \abs{\grad u}^{4-\gamma}=\, 0\quad \text{in}\; \Rd,$$
are necessarily constant. Furthermore, in Theorem~\ref{T2.5}, provided $q$ is allowed to change sign but satisfies certain decay property  at infinity, we also find another Liouville type result establishing
that any supersolution, which is bounded  below, to the equation 
$$\ginfdel u + q(x)\cdot \grad u(x) \abs{\grad u}^{2-\gamma}\,=\, 0 \quad \text{in}\; \Rd\,,$$
must be a constant. Our next result,  Theorem~\ref{T2.6},  concerns  Liouville type result for subsolution  of the equation with strong absorption
$$\ginfdel u(x) + q(x)\cdot \grad u(x) \abs{\grad u }^{2-\gamma} + c(x) (u_+(x))^\beta\,= 0
\quad \text{in}\; \Rd,$$
for $c<0$,  provided $u_+$ satisfies a suitable growth condition at infinity. This is a considerable extension of \cite[Theorem~4.4]{ALT}, which considered the case of 
$q=0$, $\gamma=0$ and
$c$ constant (see Example~\ref{E2.1} for further discussion). We remark that, another important regularity result near the boundary of the
non-coincidence set, that is,  $\partial\{u>0\}$ obtained in (cf. \cite[Theorem~4.2]{ALT}) can be 
deduced using the results of Section~\ref{S-max} considered here.

In the next step, we also study a related principal eigenvalue for the operator 
$$\sL u+c(x)u^{3-\gamma},$$
with Dirichlet boundary condition in bounded domains and use it to characterize the validity of maximum principle. This is actually a preliminary step to construct a subsolution for  equation (\ref{E1.1}) in the whole space  to be explained below. However the results can be of independent interest. Some  further insightful discussion is given by  Remark \ref{Re-eigen}.

Since the equation (\ref{E1.1}) imposed on $\Rd$, one of the main difficulties, in studying the existence and nonexistence of positive solution, is how to construct a suitable pair of sub and super-solutions. Therefore, we need to assume $q(x)$ vanishes at infinity and
\begin{equation}\label{Key-assum}
\liminf_{|x|\to\infty}\lim_{s\to0}\frac{f(x,s)}{s^{3-\gamma}}>0=\lim_{|x|\to\infty}|q(x)|.
\end{equation}
In fact, this type of condition is inspired by the series of works of Berestycki et. al. \cite{BHN,BH,BHR} in the investigation of the spreading phenomena of the transition front. In particular, Berestycki, Hamel and Rossi \cite{ BHR} considered the semilinear elliptic equation 
\begin{equation}\label{ori-equa}
\trace(A(x)D^2 u(x)) + q(x)\cdot\grad u(x) + f(x, u)=0\quad \text{in}\; \Rd,
\end{equation}
where  $f$ is of  Fisher-KPP (for Kolmogorov, Petrovsky and Piskunov) type nonlinearity, and established existence and uniqueness of  positive bounded solution under the key assumption
$$\liminf_{|x|\to\infty}\, (4\alpha(x) f_s(x, 0)-|q(x)|^2)\,>\,0\,,$$
where $\alpha(x)$ denotes the smallest eigenvalue of the matrix $A(x)$, provided $\inf_{\R^d}\alpha(x)>0$. This condition plays a central role in the
construction of a suitable subsolution \cite[Lemma~3.1]{BHR} and corresponds to our condition (\ref{Key-assum}) as $q$ vanishes at infinity. Also, note that in the degenerate case, i.e. $\alpha(x)=0$, intuitively, we should impose $q(x)\to0$ as $|x|\to\infty$.
Therefore, one of the main questions for our model is that: which suitable condition should we impose on the coefficients so that we can construct a positive solution for equation (\ref{E1.1})? We successfully solved this problem by assuming that $q$ vanishes at infinity and (\ref{Key-assum}) for 
the  equation (\ref{E1.1}). We strongly believe that this type of condition is optimal to construct the positive solution for degenerate equation such as (\ref{E1.1}). In fact, 
this claim should be compared with the interesting work of Berestycki, Hamel and Nadirashvili \cite[Theorem 1.9]{BHN} in the case $\trace(A(x)D^2 u)=\Delta u$, $q(x)=q$ being a constant and $f=f(u)$ of Fisher-KPP type. More precisely, the authors in \cite{BHN} showed that if $q>2\sqrt{f'(0)}$ then the solution of evolution equation corresponding to equation (\ref{ori-equa}) converges to zero while if $q<2\sqrt{f'(0)}$ then it converges to $1$  in the large time, which is called the vanishing/spreading phenomena. We also emphasize that  condition (\ref{Key-assum}) is sharp for existence of positive solution since  we are able to prove, in the spirit of the vanishing phenomenon as \cite[Theorem 1.9]{BHN}, the nonexistence of positive solution of equation (\ref{E1.1}) by assuming a reverse condition that 
\begin{equation}\label{neg-inf}
\limsup_{\abs{x}\to\infty} \sup_{s\in[\delta_1, \delta_2]}\, f(x, s)< 0=\lim_{|x|\to\infty}|q(x)|,\,\quad\quad\forall\delta_2>\delta_1>0.
\end{equation}
We would like to point out that this conditions are in the spirit of the conditions
used by Nguyen and Vo \cite{NV1}
 to obtain the existence and uniqueness of positive solution for quasilinear elliptic equation in the whole space. However, because of the lack of variational and linear structure of infinity Laplacian and the presence of the nonlinear gradient term most of the techniques used in \cite{BHN,BH,BHR,NV1}  cannot apply in this framework. New ideas must be figured out to deal with the current problems.

%


\textit{The paper is organized as follows : In  Section 2, we establish some preliminary results such as
comparison principle, strong maximum principle and Hopf's lemma that are used to prove the main results. Here, we also prove some direct Liouville type results without assumption at infinity on potential $c(x)$. In Section~\ref{S-eigen}, we study the related Dirichlet principal eigenvalue problem, some basic qualitative properties of the eigenvalue and  use it to characterize the maximum principle. Section~\ref{S-liouv} is devoted to proofs of the existence, nonexistence and uniqueness of positive solution of equation (\ref{E1.1}) and construction of some examples to illustrate the results.}

\section{Regularity, maximum principle, and direct Liouville results }\label{S-max}
In this section, we prove the comparison principle, strong maximum principle and Hopf's lemma which will be used throughout this article. We also develop the Liouville type results Theorems~\ref{T2.4}-\ref{T2.6}
in this section.

Let $\mathcal{O}$ be a  domain in $\Rd$. We denote $\sB_r(x)$ by  the ball of radius $r$ centered at $x$ and for $x=0$ this ball will be denoted by $\sB_r$. We use the notation $u\prec_{z} \varphi$ when $\varphi$ touches $u$
from above exactly at the point $z$ i.e., for some open ball $\sB_r(z)$ around $z$ we have $u(x)<\varphi(x)$ for $x\in\sB_r(z)\setminus\{z\}$ and $u(z)=\varphi(z)$.

To state the results in a general setting we introduce a Hamiltonian. Let $H:\bar{\mathcal{O}}\times\Rd\to \R$ be a continuous function with the following property
\begin{itemize}
\item $H(x, p)\leq C(1+|p|^\beta) \quad \text{for some}\; \beta\in(0, 3-\gamma] \; \text{and}\; (p, x)\in\Rd\times\bar{\mathcal{O}}$.
\item $\abs{H(x, p)-H(y, p)}\leq \omega(\abs{x-y}) (1+ \abs{p}^\beta)$ where 
$\omega:[0, \infty)\to[0, \infty)$ is a continuous function with $\omega(0)=0$.
\end{itemize}
In this article, we deal with the viscosity solution to the equations of the form
\begin{equation}\label{E2.1}
\ginfdel u + H(x,  \grad u) + F(x, u)\,=\, 0\quad \text{in}\;  \mathcal{O}, \quad \text{and}
\quad u=g\quad \text{on}\; \partial\mathcal{O}.
\end{equation}
Here $F$ and $g$ are assumed  to be continuous. For a symmetric matrix $A$ we define
$$M(A)=\max_{\abs{x}=1} \langle x, A x\rangle, \quad m(A)= \min_{\abs{x}=1} \langle x, A x\rangle.$$
\begin{definition}[Viscosity solution]
An upper-semicontinuous (lower-semicontinous) function $u$ on $\bar{\mathcal{O}}$ is said to be a viscosity sub-solution (super-solution) of \eqref{E2.1}  if the followings statements are satisfied:
\begin{itemize}
\item[(i)] $u\leq g$ on $\partial \mathcal{O}$ ($u\geq g$ on $\partial \mathcal{O}$);
\item[(ii)] if $u\prec_{x_0}\varphi$ ($\varphi\prec_{x_0} u$ ) for
some point $x_0\in\mathcal{O}$ and a $\cC^2$ test function $\varphi$, then 
\begin{align*}
& \ginfdel \varphi(x_0) + H(x,  \grad \varphi(x_0)) + F(x_0, u(x_0))\,\geq\, 0\,,
\\
&\left(\ginfdel \varphi(x_0) + H(x,  \grad \varphi(x_0)) + F(x_0, u(x_0))\leq\, 0,\; resp., \right);
\end{align*}
\item[(iii)] for $\gamma=2$, if $u\prec_{x_0}\varphi$ ($\varphi\prec_{x_0} u$) and $\grad\varphi(x_0)=0$ then 
\begin{align*}
& M(D^2\varphi(x_0)) + H(x,  \grad \varphi(x_0)) + F(x_0, u(x_0))\,\geq\, 0\,,
\\
&\left(m(D^2\varphi(x_0)) + H(x,  \grad \varphi(x_0)) + F(x_0, u(x_0))\leq\, 0,\; resp., \right)\,.
\end{align*}
\end{itemize}
We call $u$ a viscosity solution if it is both sub and super solution to \eqref{E2.1}.
\end{definition}
As well known, one can replace the requirement of strict maximum (or minimum) above by non-strict maximum (or minimum).
We would also require the notion of superjet and subjet from \cite{CIL}. A second order \textit{superjet} of $u$ at $x_0\in\mathcal{O}$ is defined as
$$J^{2, +}_\mathcal{O} u(x_0)=\{(\grad\varphi(x_0), D^2\varphi(x_0))\; :\; \varphi\; \text{is}\; \cC^2\; \text{and}\; u-\varphi\; \text{has a maximum at}\; x_0\}.$$
The closure of a superjet is given by
\begin{align*}
\bar{J}^{2, +}_\mathcal{O} u(x_0)&=\Bigl\{ (p, X)\in\Rd\times\bS^{d\times d}\; :\; \exists \; (p_n, X_n)\in J^{2, +}_\mathcal{O} u(x_n)\; \text{such that}
\\
&\,\qquad  (x_n, u(x_n), p_n, X_n) \to (x_0, u(x_0), p, X)\Bigr\}.
\end{align*}
Similarly, we can also define closure of a subjet, denoted by $\bar{J}^{2, -}_\mathcal{O} u$. See for instance, \cite{CIL} for more details.

Our proof of the comparison principle (Theorem~\ref{T2.1}) uses the following regularity result.
\begin{lemma}\label{L2.1}
Suppose that $u$ is a bounded solution of $\ginfdel u \geq - \theta_1 \abs{\grad u}^\beta -\theta_2 \abs{u}-\theta_3$ in $\mathcal{O}$ for some positive constants $\theta_i, i=1,2,3$ and
$\beta\in (0, 3-\gamma]$. Then 
$u$ is locally Lipschitz in $\mathcal{O}$ with (local) Lipschitz constant depending on $\theta_i$ and $\norm{u}_{L^\infty(\mathcal{O})}$.
\end{lemma}

\begin{proof}
The idea of the proof is inspired by \cite[Lemma~2.2(i)]{LW2010}. Due to Young's inequality we may choose
$\beta=3-\gamma$. 
First we note that if $\ginfdel u \geq - \theta_1 \abs{\grad u}^\beta -\theta_2 \abs{u}-\theta_3$ in $\mathcal{O}$, then
we also have
$$\infdel u \geq - \abs{\grad u}^\gamma (\theta_1 \abs{\grad u}^\beta + \theta_2 \abs{u}+\theta_3),$$
in $\mathcal{O}$, in viscosity sense. Thus a simple application of Young's inequality shows that 
$$\infdel u \geq -  (\bar\theta_1 \abs{\grad u}^{3} + \bar\theta_2 \abs{u}^{\frac{3}{3-\gamma}}+ \bar\theta_3),$$
in $\mathcal{O}$ for some $\bar{\theta}_i>0, i=1,2,3$.

Without loss of generality, we may assume that $u\geq 0$. Now we choose $\alpha>0$ small enough so that $\alpha\norm{u}_{L^\infty(\mathcal{O})}<\frac{1}{2}$. Define $w(x) = u(x) +\frac{\alpha}{2} u^2(x)$. A simple calculation yields that 
\begin{align*}
\infdel w & = (1+\alpha u)^3 \infdel u + \alpha (1+\alpha u)^2 \abs{\grad u}^4
\\
&\geq (1+\alpha u)^3\left[- \bar\theta_1 \abs{\grad u}^3 - \bar\theta_2 \abs{u}^{\frac{3}{3-\gamma}} - \bar\theta_3 + \frac{\alpha}{1+\alpha u} \abs{\grad u}^4\right].
\end{align*}
Using Minkowski's inequality, we find that 
$$\bar\theta_1 \abs{\grad u}^3\,\leq\, \frac{1}{4} (\bar\theta_1)^{4} \left[ \frac{1+\alpha u}{\alpha}\right]^{3}
+ \frac{3}{4} \frac{\alpha}{1+\alpha u} \abs{\grad u}^4.$$
Since $1\leq 1+ \alpha u<2$, we find a constant $\kappa$, depending on $\norm{u}_{L^\infty(\mathcal{O})}, \bar\theta_i, i=1,2,3,$ such that
$\infdel w\geq -\kappa$ in $\mathcal{O}$. This implies that $w$ is locally Lipschitz (cf. \cite[Lemma~2.2(i)]{LW2010}, \cite[Theorem~2.4]{BM12}).
The proof now follows by noticing that $u=\frac{1}{\alpha}(\sqrt{1+2\alpha w}-1)$.
\end{proof}

Now we prove a comparison principle in bounded domain.  It generalizes the results in \cite{ASS11,CP09}.

\begin{theorem}\label{T2.1}
Let $\mathcal{O}\subset\Rd$ be a bounded  domain
and $c, h_1, h_2\in \cC(\bar{\mathcal{O}})$. Let $F:\R\to\R$ be a continuous, strictly increasing function.
Suppose that $u\in USC(\bar{\mathcal{O}})$ is a bounded subsolution to  
\begin{equation}\label{ET2.1A}
\ginfdel u + H(x, \grad u) + c(x) F(u(x))=h_1(x)\quad \text{in} \; \mathcal{O},
\end{equation}
and $v\in LSC(\bar{\mathcal{O}})$ is a bounded super-solution to \eqref{ET2.1A} with $h_1$ replaced by $h_2$.
Furthermore, assume that $v\geq u$ on $\partial \mathcal{O}$ and one of the following holds.
\begin{itemize}
\item[(a)] $c<0$ in $\bar{\mathcal{O}}$ and $h_1\geq h_2$ in $\bar{\mathcal{O}}$.
\item[(b)] $c\leq 0$ in $\bar{\mathcal{O}}$ and $h_1> h_2$ in $\bar{\mathcal{O}}$.
\end{itemize}
Then we have $v\geq u$ in $\bar{\mathcal{O}}$.
\end{theorem}

\begin{proof}
We suppose by contradiction that $M=\max_{\bar{\mathcal{O}}}(u-v)>0$. Consider
$$w_\varepsilon(x, y)= u(x)-v(y) -\frac{1}{4\varepsilon}\abs{x-y}^4 \quad \text{for}\; x, y\in\bar{\mathcal{O}}.$$
Note that the maximum of $w_\varepsilon$ (say, $M_\varepsilon$) is bigger than $M$ for all $\varepsilon$.
Let $(x_\varepsilon, y_\varepsilon)\in\mathcal{O}\times\mathcal{O}$ be a point of maximum for $w_\varepsilon$.
It is then standard to show that (cf. \cite[Lemma~3.1]{CIL})
\begin{align*}
\lim_{\varepsilon\to 0} M_\varepsilon=M, \quad 
\lim_{\varepsilon\to 0}\frac{1}{4\varepsilon}\abs{x_\varepsilon-y_\varepsilon}^4=0.
\end{align*}
This of course, implies that $u(x_\varepsilon)-v(y_\varepsilon)\searrow M$, as $\varepsilon\to 0$. Again, since the maximizer  can not move towards the boundary
we can find a subset $\mathcal{O}_1\Subset\mathcal{O}$ such that 
$x_\varepsilon, y_\varepsilon\in\mathcal{O}_1$ for all $\varepsilon$ small. Since $u, v$ are Lipschitz continuous in $\mathcal{O}_1$, by Lemma~\ref{L2.1}, we can find a constant
$L$ such that
$$\abs{u(z_1)-u(z_2)} + \abs{v(z_1)-v(z_2)}\leq L\abs{z_1-z_2}\quad z_1, z_2\in\mathcal{O}_1.$$
Observing
$$u(x_\varepsilon)-v(x_\varepsilon)\leq u(x_\varepsilon)-v(y_\varepsilon) -\frac{1}{4\varepsilon}\abs{x_\varepsilon-y_\varepsilon}^4,$$
we obtain 
\begin{equation}\label{ET2.1B}
\abs{x_\varepsilon-y_\varepsilon}^3\,\leq\, 4\varepsilon L.
\end{equation}
Denote by $\eta_\varepsilon=\frac{1}{\varepsilon}\abs{x_\varepsilon-y_\varepsilon}^2(x_\varepsilon-y_\varepsilon)$ and $\theta_\varepsilon(x, y)=\frac{1}{4\varepsilon}\abs{x-y}^4$.
It then follows from \cite[Theorem~3.2]{CIL} that for some $X, Y\in\bS^{d\times d}$ we have $(\eta_\varepsilon , X)\in\bar{J}^{2, +}_\mathcal{O} u(x_\varepsilon)$,
$(\eta_\varepsilon, Y)\in\bar{J}^{2, -}_\mathcal{O} v(y_\varepsilon)$ and
\begin{equation}\label{ET2.1C}
\begin{pmatrix}
X & 0\\
0 & -Y
\end{pmatrix}
\leq 
D^2\theta_\varepsilon(x_\varepsilon, y_\varepsilon) + \varepsilon [D^2\theta_\varepsilon(x_\varepsilon, y_\varepsilon)]^2.
\end{equation}
In particular, we get $X\leq Y$. Moreover, if $\eta_\varepsilon=0$, we have $x_\varepsilon= y_\varepsilon$. Then from \eqref{ET2.1C} it follows that
\begin{equation}\label{ET2.1E}
\begin{pmatrix}
X & 0\\
0 & -Y
\end{pmatrix}
\leq 
\begin{pmatrix}
0 & 0\\
0 & 0
\end{pmatrix}.
\end{equation}
In particular, \eqref{ET2.1E} implies that $X\leq 0\leq Y$ and therefore, $M(X)\leq 0\leq m(Y)$.
Applying the definition of superjet and subjet we now obtain for $\eta_\varepsilon\neq 0$
\begin{align*}
h_1(x_\varepsilon)&\leq \abs{\eta_\varepsilon}^{-\gamma} \langle \eta_\varepsilon X, \eta_\varepsilon \rangle +c(x_\varepsilon) F(u(x_\varepsilon))
+H(x_\varepsilon, \eta_\varepsilon)
\\
& \leq \abs{\eta_\varepsilon}^{-\gamma} \langle \eta_\varepsilon Y, \eta_\varepsilon \rangle + c(x_\varepsilon) F(u(x_\varepsilon))
+ H(x_\varepsilon, \eta_\varepsilon)
\\
&\leq h_2(y_\varepsilon)- c(y_\varepsilon) F(v(y_\varepsilon)) + c(x_\varepsilon) F(u(x_\varepsilon))
+ H(x_\varepsilon, \eta_\varepsilon) -H(y_\varepsilon, \eta_\varepsilon)
\\
&\leq h_2 (y_\varepsilon) + F(v(y_\varepsilon)) (c(x_\varepsilon)-c(y_\varepsilon)) + [\min_{\bar{\mathcal{O}}} c] \left(F(u(x_\varepsilon))- F(v(y_\varepsilon))\right)
+\omega(\abs{x_\varepsilon-y_\varepsilon})(1+ \abs{\eta_\varepsilon}^\beta).
\end{align*}
Letting $\varepsilon\to 0$ and using \eqref{ET2.1B}, we find 
$$ [\min_{\bar{\mathcal{O}}} c]\, \sup_{s\in[-\norm{v}_{L^\infty(\mathcal{O})}, 
\norm{v}_{L^\infty(\mathcal{O})}]} (F(s+M)-F(s)) + \max_{\bar{\mathcal{O}}} (h_2-h_1)\geq 0\,.$$
This is a contradiction to (a) and (b) and thus
we prove that $u\leq v$ in $\mathcal{O}$. The argument also works when $\eta_\varepsilon=0$ and $\gamma=2$.
The result then follows.
\end{proof}
Next we prove a strong maximum principle and a Hopf's lemma.
\begin{theorem}[Strong maximum principle]\label{T2.2}
Let $\mathcal{O}$ be a bounded domain and $q$, $c$ are continuous functions in $\bar{\mathcal{O}}$. If $v\in LSC (\overline{\mathcal{O}})$ is a non-negative viscosity super-solution of
\begin{equation}\label{ET2.2A}
\ginfdel v + q(x)\cdot \grad v(x) | \grad v(x)|^{2-\gamma}+c(x)v^{3-\gamma}(x)=0, \quad x\in\mathcal{O}\,.
\end{equation}
then either $v\equiv 0$ or $v>0$ in $\mathcal{O}$. Furthermore, assume that $\mathcal{O}$ satisfies an interior
sphere condition and $v(x)>v(z)=0$ for all $x\in\mathcal{O}$ and some $z\in\partial\mathcal{O}$. Then for some
constant $\nu>0$ we have
\begin{equation}\label{ET2.2AA}
v(x)\geq\; \nu (r-\abs{x-x_0}),  \quad \text{for}\;  x\in\sB_r(x_0),
\end{equation}
where $\sB_r(x_0)\subset\mathcal{O}$ is a ball touching the point $z$.
\end{theorem}

\begin{proof}
Note that without any loss of generality we may assume that $c<0$ in $\bar{\mathcal{O}}$. Suppose that $v\gneq 0$ in $\mathcal{O}$. We show that $v>0$ in $\mathcal{O}$. On the contrary, suppose
that there exists $(x_0, r)\in\mathcal{O}\times(0, \infty)$ such that $\sB_{2r}(x_0)\Subset\mathcal{O}$, $v>0$ in $\sB_{r}(x_0)$ and $v(z)=0$ for some $z$ satisfying $\abs{x_0-z}=r$. For simplicity
we also assume that $x_0=0$. Now we construct a test function using the ideas from \cite{SP}.
Let $u(x)=e^{-\alpha |x|}-e^{-\alpha r}$. Then $u> 0$ in $\sB_r(0)$.
A straight-forward calculation shows that for $\frac{r}{2}\leq x\leq r$ we have
\begin{align}\label{ET2.2B}
&\ginfdel u(x) +q(x)\cdot \grad u(x) | \grad u(x)|^{2-\gamma}+c(x) u^{3-\gamma}(x)\nonumber
\\
&\geq \, e^{-\alpha (3-\gamma) |x|}\left[\alpha^{4-\gamma}-\norm{q}_{L^\infty(\sB_{2r})} \alpha^{3-\gamma} -\norm{c}_{L^\infty(\sB_{2r})} \left(1-e^{-\alpha(r-\abs{x})} \right)^{3-\gamma} \right]>0,
\end{align}
if we choose $\alpha$ large enough. Now choose $\kappa$ small enough so that $\kappa u\leq v$ on $\partial\sB_{\frac{r}{2}}$. This is possible since $v$ is positive on $\partial\sB_{\frac{r}{2}}$.
Thus by Theorem~\ref{T2.1} we get $v(x)\geq \kappa u(x)$ for $r/2\leq |x|\leq r$. On the other hand, $\kappa u\leq 0\leq v$ in $\sB_{2r}\setminus\sB_r$. Thus $\kappa u\prec_{z} v$ and $v(z)=0$.
By the definition of viscosity solution we must have
$$\ginfdel (\kappa u)(z) + q(z)\cdot (\kappa \grad u)(z) |\kappa \grad u(z)|^{2-\gamma}+c(z)v^{3-\gamma}(z)\leq 0\,,$$
which is a contradiction to \eqref{ET2.2B}. Therefore, we must have $v>0$ in $\Omega$. This proves the
first part of the theorem.

Also, \eqref{ET2.2AA} follows by repeating the above argument and using the fact that for any
$a>0$ there exists $\kappa>0$ satisfying $1-e^{-s}\geq \kappa s$ for all $s\in [0, a]$.
This completes the proof.
\end{proof}

Now we are ready to prove an existence result suited for our purpose.

\begin{theorem}\label{T2.3}
Let $\mathcal{O}\subset\Rd$ be a bounded $\cC^1$ domain.
Suppose that $c, h, q\in \cC(\bar{\mathcal{O}})$ and $g\in\cC(\partial\mathcal{O})$. Also, assume that $c<0$ in 
$\bar{\mathcal{O}}$. Suppose that 
$\bar u\in \cC(\bar{\mathcal{O}})$ is a super-solution and  $\underline u\in \cC(\bar{\mathcal{O}})$ is a subsolution to
\begin{equation}\label{ET2.3A}
\begin{split}
\ginfdel u + q(x)\cdot \grad u(x) \abs{\grad u(x)}^{2-\gamma} + c(x) u^{3-\gamma}(x)&=h(x)\quad \text{in} \; \mathcal{O}\,,
\\
u &= g \quad \text{on} \; \partial\mathcal{O}\,.
\end{split}
\end{equation}
with $\bar u\geq \underline u\geq 0$. We also assume that $\underline{u}=g$ on $\partial\mathcal{O}$. 
Then equation \eqref{ET2.3A} admits a unique solution $u$ satisfying $\bar u\geq u\geq \underline u\geq 0$ in $\bar{\mathcal{O}}$.
\end{theorem}

\begin{proof} The
uniqueness follows from Theorem~\ref{T2.1}.
The existence follows from standard Perron's method and construction of an appropriate barrier function at the boundary.
We sketch a proof here for completeness. Let
$$\cA=\{\psi\in LSC(\bar{\mathcal{O}})\; :\;  \psi\; \text{is a super-solution to \eqref{ET2.3A} and}\; \psi\leq\bar u\}.$$
This represents collection of all super-solutions below $\bar u$. 
Let $v(x)=\inf_{\psi\in\cA}\psi(x)$ and
$$v_*(x)=\lim_{r\to 0}\; \inf\{v(y)\; :\; y\in\bar{\mathcal{O}}, \quad \abs{x-y}\leq r\},$$
be the LSC envelope of $v$ (cf. \cite[p.~22]{CIL}). Keep in mind that by definition of super-solution, $v\geq g$ on $\partial \mathcal{O}$.
 We claim that $v_*\in\cA$ which would then imply $v=v_*$. By Lemma~\ref{L2.1} it follows that $v$ is locally Lipschitz continuous
in $\mathcal{O}$. Now suppose that $v_*$ is not a super-solution. Suppose that for some $\varphi\in\cC^2(\mathcal{O})$ such that $\varphi\prec_{x_0} v_*$ for some $x_0\in\mathcal{O}$ and
$$\ginfdel\varphi(x_0)+ q(x_0)\cdot\grad\varphi(x_0) \abs{\grad \varphi(x_0)}^{2-\gamma}+
c(x)\varphi^{3-\gamma}(x_0)\,>\, h(x_0).$$
Using continuity we can find a ball $\sB(x_0)$ around $x_0$ satisfying 
\begin{equation}\label{ET2.3B}
\ginfdel\varphi(x)+ q(x)\cdot\grad\varphi(x)\abs{\grad \varphi(x)}^{2-\gamma} +c(x)\varphi^{3-\gamma}(x)\,> h(x)\quad\quad\textrm{in $\overline{\sB(x_0)}$}.
\end{equation}
Now, for every $\varepsilon>0$ we can find a pair $(\psi_\varepsilon, x_\varepsilon)\in \cA\times\sB(x_0)$ satisfying
$$\psi_\varepsilon(x_\varepsilon)- \varphi(x_\varepsilon)=\inf_{\sB(x_0)} (\psi_\varepsilon-\varphi)<\varepsilon.$$
Note that $x_\varepsilon\to x_0$ as $\varepsilon\to 0$. Also, $\varphi + \psi_\varepsilon(x_\varepsilon)- \varphi(x_\varepsilon)$ touches $\psi_\varepsilon$ at $x_\varepsilon$ from below.
Thus by the definition of super-solution we must have 
$$\ginfdel \varphi(x_\varepsilon) +q(x_\varepsilon)\cdot\grad\varphi(x_\varepsilon)
\abs{\grad \varphi(x_\varepsilon)}^{2-\gamma} + 
c(x_\varepsilon)\psi^{3-\gamma}_{\varepsilon}(x_\varepsilon)\,\leq h(x_\varepsilon),$$
and letting $\varepsilon\to 0$, we obtain a contradiction to \eqref{ET2.3B}. To complete the claim it remains to show that $v_*\geq g$ on $\partial \mathcal{O}$. This follows from the fact that 
$\psi\geq \underline{u}$ for all $\psi\in\cA$, by Theorem~\ref{T2.1}. Thus, $v\geq \underline{u}$ and
$v_*\geq g$ on $\partial \mathcal{O}$.

It also standard to show that $v$ is a sub-solution in $\mathcal{O}$. For instance, we can follow the arguments in \cite[Theorem~1]{LW2008}. To complete the
proof we must check that $\lim_{x\in \mathcal{O}\to z} v(x)=g(z)$ for all $z\in\mathcal{O}$. Pick $z\in\partial \mathcal{O}$. We consider a continuous extension of $g$ to $\Rd$.
For a given $\varepsilon>0$ we choose
$r>0$ such that $|g(x)-g(z)|\leq \varepsilon $ for $x\in \sB_r(z)$.
We construct a barrier now. Fix any $r_1\in (0, r\wedge 1)$ and define the function
$\chi(x)= (\abs{x}^\alpha-r_1^{\alpha})$ for $\alpha\in (0, 1)$. A direct calculation yields
\begin{align*}
\ginfdel \chi(x) + \norm{q}_{L^\infty(\mathcal{O})} \abs{\grad\chi(x)}^{3-\gamma}
=\alpha^{3-\gamma} \abs{x}^{(3-\gamma)\alpha-4+\gamma} (\alpha-1) + 
\norm{q}_{L^\infty(\mathcal{O})} \alpha^{3-\gamma} \abs{x}^{(3-\gamma)(\alpha-1)} 
\end{align*}
for $|x|>r_1$. Since 
$$(3-\gamma)\alpha-4+\gamma=3\alpha-4 - \gamma(\alpha-1)< (3\alpha-3)
- \gamma(\alpha-1)= (3-\gamma)(\alpha-1)<0,$$
 we can choose $r_1, \delta>0$ small so that for $r_1\leq |x|\leq r_1+\delta$ we have
$$\ginfdel (\kappa\chi(x)) + \norm{q}_{L^\infty(\mathcal{O})} \abs{\kappa\grad\chi(x)}^{3-\gamma}+c(x) \kappa\chi(x)\leq -\norm{h}_{L^\infty(\mathcal{O})},$$
for some large $\kappa$. If required, we can rotate the domain such that $\sB_{r_1}(0)$ touches $\mathcal{O}$ from outside at $z$.
Let $w= \min\{g(z) + \varepsilon + \kappa \chi, \bar u\}$. Then this is also a super-solution, and thus $v\leq w$ in $\mathcal{O}$ which gives
us $\lim_{x\to z} v(x)\leq w(z)\leq g(z) +\varepsilon$. The arbitrariness of $\varepsilon$ confirms the proof.
\end{proof}

We end this section with three Liouville type results. Recall that the original Liouville property of infinity Laplacian stated any  super-solutions
$u$ of $-\infdel u=0$ in $\Rd$, which are bounded below are necessarily  constant \cite{CEG,TB01,BD01,Lindqvist}. We extend this result in Theorem~\ref{T2.4} and ~\ref{T2.5} below.
The proofs of these results are intrinsically based on  Theorem~\ref{T2.1}.
\begin{theorem}[Liouville property I]\label{T2.4}
If $u$ is a locally Lipschitz viscosity solution to $\ginfdel u -c \abs{\grad u}^{4-\gamma}\leq 0$ in $\Rd$, where $c\geq 0$ and $\inf_{\Rd}  u>-\infty$, then $u$ is necessarily a constant.
\end{theorem}

\begin{proof}
Let $\theta_\varepsilon(x) = \frac{|x|^\alpha}{\varepsilon^\alpha}$ for $ \alpha\in (-1, 0)$. Then a direct calculation shows that for $\abs{x}>0$
$$\infdel \theta_\varepsilon -\abs{\grad \theta_\varepsilon}^4 = \alpha^3 \varepsilon^{-3\alpha} \abs{x}^{3\alpha-4} \left(\alpha-1 + \varepsilon^{-\alpha}\abs{x}^{\alpha}\right).$$
Thus for $\abs{x}\geq \varepsilon$ we have $\infdel \theta_\varepsilon -\abs{\grad \theta_\varepsilon}^4>0$, since $\alpha<0$. Hence,  we have 
$$\infdel \theta_\varepsilon -\abs{\grad \theta_\varepsilon}^4\,>0
\quad \text{in}\; \sB^c_\varepsilon(0).$$
It is easily seen that
$$\infdel u -c \abs{\grad u}^{4}\,\leq\, 0.$$
If $c=0$, then it becomes the standard Liouville property for super-harmonic functions. So we assume that $c>0$.
Without loss of generality we may assume that $c=1$. Otherwise, replace $u$ by $cu$. Also, by translating $u$ we may
assume that $\inf u=\frac{1}{2}$. We need to show that $u\equiv\frac{1}{2}$. 
Suppose, on the contrary, that there exists a point $x_0$ satisfying $\frac{1}{2}<u(x_0)<1$. 
With no loss of generality we may assume $x_0=0$. Let $m_\varepsilon=\min_{\sB_\varepsilon} u$. It is evident that for all $\varepsilon$ small we have
$\frac{1}{2}<m_\varepsilon<1$. For some large $R$ consider the domain $\mathcal{O}_R=\sB_R\setminus \overline{\sB_\varepsilon}$. Since $\theta_\varepsilon(x)\to 0$
as $\abs{x}\to \infty$, it follows that $u\geq m_\varepsilon \theta_\varepsilon$ on $\partial\mathcal{O}_R$. Applying Theorem~\ref{T2.1} we find that
$u\geq m_\varepsilon \theta_\varepsilon$ in $\mathcal{O}_R$ (note that value of $\gamma$ is not important in Theorem~\ref{T2.1} if the solutions are locally Lipschitz and the same proof works since $\theta_\varepsilon$ is
a strict subsolution). Now letting $R\to\infty$ we obtain that $u\geq m_\varepsilon \theta_\varepsilon$ in $\sB^c_\varepsilon$ for all $\alpha<0$. Let $\alpha\to 0$ to find that
$u(x)\geq m_\varepsilon$ for any $|x|\geq \varepsilon$. This of course, implies that $\inf_{\Rd} u \geq m_\varepsilon>\frac{1}{2}$ which is a contradiction. 
Hence, the proof is complete.
\end{proof}

Our next Liouville result includes $q$ with certain decay rate at infinity. This  function
should be seen as the \textit{small perturbation} to the infinity Laplacian.
\begin{theorem}[Liouville property II]\label{T2.5}
Suppose that 
$$\limsup_{\abs{x}\to\infty}\, (q(x)\cdot x)_+<1.$$
Then every  solution $u$ to $\ginfdel u + 
q(x)\cdot \grad u(x) \abs{\grad u(x)}^{2-\gamma}\leq 0$ in $\Rd$, with $\inf_{\Rd} u>-\infty$, is necessarily  a constant.
\end{theorem}

\begin{proof}
As earlier, we can write
$$ \infdel u + q(x)\cdot\grad u(x) \abs{\grad u(x)}^2\leq \, 0\quad \text{in}\; \Rd.$$
We claim that there exists a compact set $K$ such that 
\begin{equation}\label{ET2.5A}
\inf_{\Rd} u =\min_{K} u\,.
\end{equation}
Let us first show that  the claim \eqref{ET2.5A} implies $u$ to be a constant. By \eqref{ET2.5A}, the function $v(x)=u(x)-\min_{K} u$ is a non-negative solution to
$$ \infdel v + q(x)\cdot\grad v(x) \abs{\grad v(x)}^2\leq \, 0\quad \text{in}\; \Rd,$$
and $v$ vanishes somewhere in $K$. Applying Theorem~\ref{T2.2} we obtain that $v\equiv 0$ and therefore, $u$ is a constant.

Now we prove the claim \eqref{ET2.5A}. By translating $u$ we may assume that $u\geq 1$ in $\Rd$.
Let $K$ (containing $0$ in the interior) be such that $(q(x)\cdot x)_+<1$ for $x\in K^c$. Let $\theta_\alpha(x) =\abs{x}^\alpha$ for $\alpha\in (-1, 0)$. A
routine calculation reveals that for $x\in K^c$
\begin{align*}
\infdel\theta + q(x)\cdot\grad\theta(x) \abs{\grad\theta(x)}^2 &= \alpha^3 (\alpha-1) \abs{x}^{3\alpha-4} + \alpha^3 q(x)\cdot x \abs{x}^{3\alpha-4}
\\
&= \alpha^3 \abs{x}^{3\alpha-4} \left((\alpha-1) + q(x)\cdot x \right)
\\
&= \alpha^3 \abs{x}^{3\alpha-4} \left(\alpha + \left((q(x)\cdot x)_+ -1\right) - (q(x)\cdot x)_- \right) >0\,.
\end{align*}
Thus, for all large $R$, we have $u\geq [\min_{K} u]\delta_\alpha\,\theta_\alpha$ in $\sB_R\setminus K$, by Theorem~\ref{T2.1}, where $\delta_\alpha=[\max_{\partial K}\theta_\alpha]^{-1}$. Letting $R\to\infty$ we get 
$u(x)\geq [\min_{K} u]\delta_\alpha\,\theta_\alpha(x)$ for $x\in K^c$ and $\alpha<0$. Now let $\alpha\to 0$ to conclude \eqref{ET2.5A}.
\end{proof}

\begin{remark}
The decay of $q$ in the above Liouville property seems optimal. For example, take $q(x)=\frac{4x}{1+x^2}$
and let $u(x)=\frac{1}{1+x^2}$ for $x\in\R$. Then a straightforward calculation reveals
$$\infdel u + q(x) (u^\prime)^3= -\frac{8x^2}{(1+x^2)^6}\leq 0\quad \text{in}\; \Rd.$$
Note that $\lim_{|x|\to\infty} (q(x)\cdot x)_+=4$.
\end{remark}

The next Liouville result generalizes \cite[Theorem~4.4]{ALT}. 
\begin{theorem}[Liouville III]\label{T2.6}
Suppose that $\Lyap:\Rd\to[0, \infty)$ is a locally Lipschitz function with $\liminf_{\abs{x}\to\infty}\Lyap(x)>0$ and satisfies
\begin{equation}\label{ET2.6A}
\ginfdel \Lyap + q(x)\cdot \grad \Lyap(x) \abs{\grad \Lyap }^{2-\gamma} + c(x) \Lyap^\beta(x)\leq 0\quad \text{in}\; \Rd\,,
\end{equation}
where $c<0$ is a continuous function and $\beta\in (0, 3-\gamma]$. Let $u\in\cC(\Rd)$ be a function satisfying 
\begin{equation}\label{ET2.6B}
\ginfdel u(x) + q(x)\cdot \grad u(x) \abs{\grad u }^{2-\gamma} + c(x) (u_+(x))^\beta\,\geq 0\quad \text{in}\; \Rd\,.
\end{equation}
If we have
\begin{equation}\label{ET2.6C}
\lim_{\abs{x}\to \infty} \, \frac{u_+(x)}{\Lyap(x)}\,=\,0\,,
\end{equation}
then $u\leq  0$ (in particular, if $u\geq 0$, then $u\equiv 0$).
In addition, if we also assume that
$$\limsup_{\abs{x}\to\infty}\, (q(x)\cdot x)_+<1,$$
then $u$ is a constant.
\end{theorem}

\begin{proof}
For $\kappa\in (0, 1]$, we define $\Lyap_\kappa(x)=\kappa\Lyap$. Since $c<0$ and $\beta\in (0, 3-\gamma]$ it follows from \eqref{ET2.6A} that
\begin{equation}\label{ET2.6D}
\ginfdel \Lyap_\kappa + q(x)\cdot \grad \Lyap_\kappa(x) \abs{\grad \Lyap_\kappa }^{2-\gamma} + c(x) \Lyap^\beta_\kappa(x)\leq 0\quad \text{in}\; \Rd\,.
\end{equation}
Now using \eqref{ET2.6C} we can find $R_0>0$ such that $u\leq \Lyap_k$ on $\sB_R$ for every $R\geq R_0$. Then, by the proof of Theorem~\ref{T2.1}, and \eqref{ET2.6B} and \eqref{ET2.6D}, it follows that 
$u\leq \Lyap_k$ in $\sB_R$ for all $R\geq R_0$. Letting $R\to\infty$, we obtain $u\leq \kappa \Lyap$ for all $\kappa\in(0, 1]$. The first part  follows by letting $\kappa\to 0$.

For the second part, we observe that $\tilde{u}= -\min\{u, 0\}$ satisfies
$$\ginfdel \tilde{u} +  q(x)\cdot \grad\tilde{u}(x) \abs{\grad\tilde{u}}^{2-\gamma}\,\leq\, 0\quad \text{in}\; \Rd.$$
Therefore, the result follows from Theorem~\ref{T2.5}.
\end{proof}

We remark that the case $\gamma=0$ and $q\equiv 0$ corresponds to \cite[Theorem 4.4]{ALT}. Below we give a family of operators satisfying \eqref{ET2.6A}. 
\begin{example}\label{E2.1}
Let $\gamma\in [0, 2]$ and $\beta \in (0, 3-\gamma]$. Fix $\alpha=\frac{4-\gamma}{3-\gamma-\beta}>1$ and $\Lyap(x)=\abs{x}^\alpha$. From the computations in the proof of Theorem~\ref{T2.5}, it follows that
\begin{align*}
&\ginfdel\Lyap(x) + q(x)\cdot\grad\Lyap(x) \abs{\grad\Lyap(x)}^{2-\gamma} 
\\
&= \alpha^{3-\gamma} (\alpha-1) \abs{x}^{3\alpha-4-\gamma(\alpha-1)} + \alpha^{3-\gamma} q(x)\cdot x \abs{x}^{3\alpha-4-\gamma(\alpha-1)}
\\
&\le\, |x|^{\alpha\beta} \alpha^{3-\gamma}((\alpha-1) + (q(x)\cdot x)_+).
\end{align*}
Letting $-c(x)\geq \alpha^{3-\gamma}((\alpha-1) + (q(x)\cdot x)_+)$ it follows from above that
$$ \ginfdel\Lyap(x) + q(x)\cdot\grad\Lyap(x) \abs{\grad\Lyap(x)}^{2-\gamma} + c(x) \Lyap^\beta(x)\leq 0\quad \text{in}\; \Rd.$$
Since $\grad\Lyap(0)=0$, it is straightforward to check that $\Lyap$ is a solution to the above equation for $\gamma\in [0, 2)$. For $\gamma=2$, we note that $\alpha>2$, and therefore, for any
$\varphi\prec_0\Lyap$ we have $D^2\varphi(0)\leq 0$. Thus $m(D^2\varphi(0))\leq 0$, implying
$\Lyap$ to be a solution to the above inequality.
\end{example}

\section{Dirichlet principal eigenvalue problem}\label{S-eigen}

Let $\mathcal{O}$ be a bounded $\cC^1$ domain. The goal of this section is to prove the existence of a principal eigenfunction and other related properties.
The functions $q, c$ are assumed to be continuous in $\bar{\mathcal{O}}$. We denote by $\sL$ the operator
$$\sL \varphi(x) \,=\,  \ginfdel \varphi(x) + 
q(x)\cdot \grad\varphi(x) \abs{\grad\varphi(x)}^{2-\gamma}.$$
The principal eigenvalue of $\sL+c$ is defined as follows
$$\lambda_\mathcal{O}=\sup\{\lambda\in\R \; :\;\exists \; \varphi\in \cC(\bar{\mathcal{O}})\;
\mbox{satisfying}\;
 \sL\varphi + c(x)\varphi^{3-\gamma} + \lambda \varphi^{3-\gamma}
\leq0 \quad \text{in}\; \mathcal{O}, \; \text{and}\; \min_{\bar{\mathcal{O}}}\varphi>0\}.$$

\begin{remark}
It should be noted that the definition of principal eigenvalue is different from the one appeared in \cite{BNV}. In the above we consider
only those super-solutions that are positive in $\bar{\mathcal{O}}$ whereas the class of super-solutions considered in \cite{BNV} might vanish on $\partial\mathcal{O}$. The above
definition is similar to the one considered in \cite{BM13,PJ07}.
\end{remark}

It is clear that $\lambda_\mathcal{O}\geq -\norm{c}_\infty$. Lemma~\ref{L3.3} below shows that 
$\lambda_\mathcal{O}$ is finite.
Our goal is to prove existence of an eigenfunction associated with $\lambda_\mathcal{O}$. To do so, we need some intermediate results. The first one is about the boundary behavior.

\begin{lemma}\label{L3.1}
Let $h$ be bounded. Then for any
positive solution of $\sL u + c(x)u^{3-\gamma} = h$ that vanishes on $\partial\mathcal{O}$, we have for any $\alpha\in (0, 1)$, 
that
$$\abs{u(x)}\leq C \dist^\alpha (x, \partial\mathcal{O}),$$
for some constant $C$, depending on $\norm{u}_{L^\infty(\mathcal{O})}, \alpha$. In particular, if 
$\norm{u}_{L^\infty(\mathcal{O})}\leq 2$, then the constant $C$ can be chosen independent of $u$.
\end{lemma}

\begin{proof}
Suppose $\norm{u}_{L^\infty(\mathcal{O})}\leq k\in (1, \infty)$. Then we note that 
$$\ginfdel u + q(x)\cdot \grad u(x) \abs{\grad u}^{2-\gamma} + 
(c(x)-k\norm{c}_{L^\infty(\mathcal{O})})u^{3-\gamma} \geq  h-k^{4-\gamma} 
\norm{c}_{L^\infty(\mathcal{O})}.$$
Let $r_0$ be the radius of exterior sphere of $\mathcal{O}$.  Pick any point $z\in\partial\mathcal{O}$. Fix any $r_1\in (0, r_0\wedge 1)$ and define the function
$\chi(x)= (\abs{x}^\alpha-r_1^{\alpha})$ for $\alpha\in (0, 1)$. A direct calculation yields
\begin{align*}
\ginfdel \chi(x) + \norm{q}_{L^\infty(\mathcal{O})} \abs{\grad\chi(x)}^{3-\gamma} =\alpha^{3-\gamma} \abs{x}^{(3-\gamma)\alpha-4+\gamma} (\alpha-1) + \norm{q}_{L^\infty(\mathcal{O})} \alpha^{3-\gamma} \abs{x}^{(3-\gamma)(\alpha-1)} 
\end{align*}
for $|x|>r_1$. Since 
$$(3-\gamma)\alpha-4+\gamma=3\alpha-4 - \gamma(\alpha-1)< (3\alpha-3)
- \gamma(\alpha-1)= (3-\gamma)(\alpha-1)<0,$$
we can choose $r_1, \delta>0$ small enough so that for $r_1\leq |x|\leq r_1+\delta$ we have
$$\ginfdel (\kappa\chi(x)) + \norm{q}_{L^\infty(\mathcal{O})} \abs{\kappa\grad\chi(x)}^{2-\gamma}
+ (c(x)-k\norm{c}_{L^\infty(\mathcal{O})}) (\kappa\chi)^{3-\gamma}(x)
\leq h-k^4 \norm{c}_{L^\infty(\mathcal{O})},$$
for some large $\kappa$. $\kappa$ can be chosen large enough to satisfy $\kappa\chi(x)\geq k$ for $|x|=r_1+\delta$. Let $\sB_{r_1}(z_0)$ be the ball that
touches $\mathcal{O}$ from outside at the point $z$. Define $\tilde\chi(x)=\kappa \chi(x-z_0)$ in $\sB_{r_1+\delta}(z_0)$. Then by comparison
principle we have $u\leq \tilde\chi$ in $\mathcal{O}\cap(\sB_{r_1+\delta}(z_0)\setminus\sB_{r_1+\delta}(z_0)).$ Since $z$ is arbitrary this proves the result for $u$.  This completes the proof.
\end{proof}

Next we prove a maximum principle.
\begin{lemma}\label{L3.2}
Suppose that $\lambda<\lambda_\mathcal{O}$. Then for any solution $\sL u(x) + c(x)u_+^{3-\gamma} + \lambda u_+^{3-\gamma} \geq 0$ in $\mathcal{O}$ with $u\leq 0$ on $\partial\mathcal{O}$
we must have $u\leq 0$ in $\mathcal{O}$.
\end{lemma}

\begin{proof}
By definition we can find a $\lambda_1\in (\lambda, \lambda_\mathcal{O})$ and $v\in\cC(\bar{\mathcal{O}})$, positive on the boundary, satisfying
$$\sL v + c(x)v^{3-\gamma} + \lambda_1 v^{3-\gamma} \leq 0 \quad \text{in}\; \mathcal{O}.$$
Suppose $u^+\neq 0$. Let $\kappa=\max_{\mathcal{O}}\frac{u}{v}>0$. Then $\kappa v$ touches $u$ from above in $\mathcal{O}$. We replace $v$ by $\kappa v$ in the above.
Repeating the arguments of Theorem~\ref{T2.1}, we can find a point $z\in\mathcal{O}$ with $v(z)=u(z)>0$ and $\lambda u^{3-\gamma}(z) \geq \lambda_1 v^{3-\gamma}(z)$, which  contradicts 
the choice of $\lambda_1$. The result thus follows.
\end{proof}

Now we can show that $\lambda_\mathcal{O}$ is finite. Due to the monotonicity property with respect to domains it is enough to show that it is finite over balls.

\begin{lemma}\label{L3.3}
Let $\lambda_R$ be the principal eigenvalue in $\sB_R(0)$. Then it holds that $\lambda_R<\infty$.
\end{lemma}

\begin{proof}
Let $\varphi(x)=(e^{-k\abs{x}^2}-e^{-k\abs{R}^2})$. We show that for some large $k$ and $\lambda$ we have 
$$\ginfdel \varphi + q(x)\cdot \grad\varphi \abs{\grad\varphi}^{2-\gamma} + (c(x)+\lambda)\varphi^{3-\gamma} \,>\, 0\quad \text{in}\; \sB_R(0).$$
Then, by Lemma~\ref{L3.2}, it follows that $\lambda_R \leq \lambda$, which proves the result. A direct computation yields
\begin{align*}
&\ginfdel \varphi + q(x)\cdot \grad\varphi \abs{\grad\varphi}^{2-\gamma} + (c(x)+\lambda)\varphi^{3-\gamma}
\\
&\geq (2k\abs{x})^{4-\gamma}  e^{-(3-\gamma)k\abs{x}^2} - (2 k)^{3-\gamma} \abs{x}^{2-\gamma} e^{-(3-\gamma)k\abs{x}^2} -\norm{q}_\infty (2k\abs{x})^{3-\gamma}  e^{-(3-\gamma)k\abs{x}^2} 
\\
&\qquad + (c(x)+\lambda) \left( e^{-k\abs{x}^2}-e^{-k\abs{R}^2}\right)^{3-\gamma}
\\
&\geq e^{-(3-\gamma)k\abs{x}^2} \Bigl[(2k\abs{x})^{4-\gamma} - (2 k)^{3-\gamma} \abs{x}^{2-\gamma} -\norm{q}_\infty (2k\abs{x})^{3-\gamma} 
\\
&\qquad + (\lambda-\norm{c}) \left(1-e^{-k(\abs{R}^2-\abs{x}^2)}\right)^{3-\gamma}\Bigr].
\end{align*}
Now choose $k$ large enough so that for $R/2\leq \abs{x}\leq R$ we have
$$(2k\abs{x})^{4-\gamma} - (2 k)^{3-\gamma} \abs{x}^{2-\gamma} -\norm{q}_\infty (2k\abs{x})^{3-\gamma}\,>\, 0\,.$$
 With this choice of 
$k$, we choose $\lambda>\norm{c}$ large so that for $|x|\leq R/2$ we get 
$$(2k\abs{x})^{4-\gamma} - (2 k)^{3-\gamma} \abs{x}^{2-\gamma} -\norm{q}_\infty (2k\abs{x})^{3-\gamma} + (\lambda-\norm{c}) \left(1-e^{-k(\abs{R}^2-\abs{x}^2)}\right)^{3-\gamma}>0.$$
Note that for $\gamma=2$ we have to modify the calculation at $x=0$, but the estimate holds. 
Combining both cases together we have the result.
\end{proof}

Let us now prove a standard existence result.

\begin{lemma}\label{L3.4}
Let $\lambda<\lambda_\mathcal{O}$. Then there exists a positive solution to $\sL u + c(x)u^{3-\gamma} + \lambda u^{3-\gamma} = -1$ in $\mathcal{O}$ with $u=0$ on $\partial\mathcal{O}$.
\end{lemma}

\begin{proof}
Note that $0$ is a subsolution. Also for any $\lambda_1\in (\lambda, \lambda_\mathcal{O}, )$ there exists 
$v\in\cC(\bar{\mathcal{O}})$, positive on the boundary, satisfying
$$\sL v + c(x)v^{3-\gamma} + \lambda_1 v^{3-\gamma}\,\leq\, 0,\quad \text{in}\; \mathcal{O}.$$
Thus, for large $\kappa>1$, we have
$$\sL (\kappa v) + c(x) (\kappa v)^{3-\gamma} \leq -\lambda_1 (\kappa v)^{3-\gamma}= - \lambda (\kappa v)^{3-\gamma} + (-\lambda_1+\lambda) (\kappa v)^{3-\gamma} 
\leq -\lambda (\kappa v)^{3-\gamma} -1.$$
This gives a super-solution. Then existence follows from the monotone iteration method and comparison principle in Theorem~\ref{T2.1}. Note that the sequence of monotone iteration
function converges due to Lemma~\ref{L2.1} and Lemma~\ref{L3.1}.
The strict positivity follows from the 
strong maximum principle, Theorem~\ref{T2.2}.
\end{proof}
Now we prove an existence of a principal eigenpair.

\begin{theorem}\label{T3.1}
There exists a positive solution $\varphi$ of $\sL\varphi + c(x)\varphi^{3-\gamma} + \lambda_\mathcal{O} \varphi^{3-\gamma}=0$ in $\mathcal{O}$ with $\varphi=0$ on $\partial\mathcal{O}$.
\end{theorem}

\begin{proof}
Assume $-\lambda_\mathcal{O}$ is positive, otherwise translate.
Let $(\psi_n, \lambda_n)$ be a sequence of solutions from Lemma~\ref{L3.4} and $\lambda_n\searrow \lambda_\mathcal{O}$. We claim that $\norm{\psi_n}_{L^\infty(\mathcal{O})}$ is unbounded. 
If not, employing 
Lemma~\ref{L2.1} and Lemma~\ref{L3.1}, we can find a  subsequence of $\{\psi_n\}$, converging to $\psi$ with $\norm{\psi}_{L^\infty(\mathcal{O})}>0$, and 
$$\sL\psi + c(x)\psi^{3-\gamma} = -1- \lambda_\mathcal{O} \psi^{3-\gamma}\quad \text{in}\; \mathcal{O}, \quad \text{and}\quad \psi=0\quad \text{on}\; \partial\mathcal{O}.$$
Again, by Theorem~\ref{T2.2}, $\psi>0$ in $\mathcal{O}$. Note that for $\psi_\varepsilon=\psi + \varepsilon$
$$\sL\psi_\varepsilon + c(x)\psi_\varepsilon^{3-\gamma} \leq -1- \lambda_\mathcal{O} \psi^{3-\gamma} +\order(\varepsilon)
\leq \left[-\frac{1+\order(\varepsilon)}{\max\psi_\varepsilon} -\lambda_\mathcal{O}\right] \psi_\varepsilon^{3-\gamma}  \leq -\mu \psi_\varepsilon^3,$$
for some $\mu>\lambda_\mathcal{O}$, provided we choose $\varepsilon$ small enough. This contradicts to the definition of $\lambda_\mathcal{O}$ and this
confirms the claim. Now define $\varphi_n=[\norm{\psi_n}_{L^\infty(\mathcal{O})}]^{-1}\psi_n$. 
Then use Lemma~\ref{L2.1} and Lemma~\ref{L3.1}, to pass to the limit and obtain a principal eigenfunction.
\end{proof}

\begin{corollary}\label{C3.1}
In view of the above Theorem~\ref{T3.1} and Lemma~\ref{L3.2} we obtain the following characterization of the principal eigenvalue
$$\lambda_\mathcal{O}=\inf\{\lambda\in\R\; :\; \exists \; \psi^+\neq 0, \; \psi\leq 0\; \text{on}\; \partial \mathcal{O}, \; \text{and}\; \sL\psi + c(x)\psi_+^{3-\gamma} + \lambda\psi_+^{3-\gamma}\geq 0 \}.$$
\end{corollary}

In view of the above corollary the following result is immediate
\begin{lemma}\label{L3.5}
Suppose that $\lambda_\alpha$ is the principal eigenvalue with respect to the potential $\alpha c(x)$. Then it holds that $\lim_{\alpha\to\infty} \frac{\lambda_\alpha}{\alpha}=-\sup_{\mathcal{O}} c$.
\end{lemma}

\begin{proof}
It is obvious from the definition that $\liminf_{\alpha\to\infty} \frac{\lambda_\alpha}{\alpha}\geq -\sup_{\mathcal{O}} c$. So we only prove the following
\begin{equation}\label{EL3.5A}
\limsup_{\alpha\to\infty} \frac{\lambda_\alpha}{\alpha}\leq -\sup_{\mathcal{O}} c.
\end{equation}
Pick $\varepsilon>0$ and consider a ball $\sB\Subset \mathcal{O}$ such that $c(x)> \sup_{\mathcal{O}} c -\varepsilon$ in $\bar\sB$. Let $(\tilde\lambda_\sB, \varphi_\sB)$ be the principal 
eigenpair in $\sB$ for the operator $\sL$. Then for large $\alpha$ it holds that
\begin{align*}
\sL\varphi_\sB + \alpha c(x) \varphi^{3-\gamma}_\sB 
& \geq -\tilde\lambda_\sB \varphi^{3-\gamma}_\sB + \alpha\, (\sup_{\mathcal{O}} c -\varepsilon) \varphi^{3-\gamma}_\sB
\\
&= (-\tilde\lambda_\sB+\alpha\varepsilon)\varphi^{3-\gamma}_\sB + \alpha\, (\sup_{\mathcal{O}} c -2\varepsilon) \varphi^{3-\gamma}_\sB
\\
&\geq \alpha\, (\sup_{\mathcal{O}} c -2\varepsilon) \varphi^{3-\gamma}_\sB.
\end{align*}
Hence, by Corollary~\ref{C3.1}, we find $\lambda_\alpha\leq \lambda_{\sB}(\alpha c) \leq -\alpha\, (\sup_{\mathcal{O}} c -2\varepsilon)$ which in turn, gives
$$\limsup_{\alpha\to\infty} \frac{\lambda_\alpha}{\alpha}\,\leq\, -\sup_{\mathcal{O}} c + 2\varepsilon.$$
Since $\varepsilon$ is arbitrary, we get \eqref{EL3.5A}. Hence, the proof is complete.
\end{proof}

Let us also prove a continuity property of $\lambda_\mathcal{O}$ with respect to decreasing domains.

\begin{lemma}\label{L3.6}
Let $\{\mathcal{O}_n\}$ be  an exhaustion of $\mathcal{O}$, i.e $\mathcal{O}_n$ is  smooth bounded domain with uniform radius of exterior sphere such that $\mathcal{O}_{n+1}  \Subset\mathcal{O}_n$ and $\bigcap_n\mathcal{O}_n=\mathcal{O}$. Then 
we have $\lambda_{\mathcal{O}_n}\to\lambda_\mathcal{O}$.
\end{lemma}

\begin{proof}
Let $(\lambda_{\mathcal{O}_n}, \varphi_n)$ be a principal eigenpair obtained in Theorem~\ref{T3.1}. Also, set $\norm{\varphi_n}=1$. Since all the domains has a uniform radius of  exterior spheres, 
the constant $C$ in Lemma~\ref{L3.1} can be chosen independent of the domains. Therefore, employing Lemma~\ref{L2.1} and Lemma~\ref{L3.1}, we can extract a subsequence of 
$\varphi_n$ converges uniformly to $\varphi$ with $\norm{\varphi}=1$, $\varphi\geq 0$ and 
$$\ginfdel \varphi + q(x)\cdot \grad \varphi \abs{\grad \varphi}^{2-\gamma} + c(x)\varphi^{3-\gamma} + \lambda\varphi^{3-\gamma}\,=0\quad \text{in}\; \mathcal{O}, \quad \text{and}\quad \varphi=0\quad \text{on}\; \partial\mathcal{O},$$
where $\lambda=\lim_{n\to\infty}\lambda_{\mathcal{O}_n}\leq \lambda_\mathcal{O}$. We also imply, by strong maximum principle (Theorem~\ref{T2.2}), that $\varphi>0$ in $\mathcal{O}$. It then follows from 
Corollary~\ref{C3.1} that
$\lambda=\lambda_\mathcal{O}$. This completes the proof.
\end{proof}


\section{Bounded positive solutions of $\sL u(x) + f(x, u)=0$}\label{S-liouv}
The main goal of the section is to prove Theorem~\ref{T4.1}  which
gives existence of a unique positive solution to $\sL u + f(x, u)=0$ in $\Rd$.
 Recall that 
$$\sL \varphi(x) \,=\,  \ginfdel \varphi(x) + q(x)\cdot \grad\varphi(x) \abs{\grad\varphi(x)}^{2-\gamma}.
$$
Throughout this section we make the following assumption on $q$.
\begin{itemize}
\item[(Q)] $q:\Rd\to\Rd$ is continuous and vanishing at infinity.
\end{itemize}
Let $f:\Rd\times[0, \infty)\to\R$ be a continuous function with the following properties:
\begin{enumerate}
\item[(A1)] $f:\Rd\times[0, \infty)\to \R$ is continuous and $f(x, 0)=0$. Also, $f(x, \cdot)$ is locally Lipschitz in $[0, \infty)$ uniformly with respect to $x$.
\item[(A2)]  For some $M>0$ we have $f(x, M)\leq 0$ in $\Rd$.
\item[(A3)] The limit $\ell(x)\df\lim_{s\to 0} \frac{f(x, s)}{s^{3-\gamma}}$ exists  uniformly with respect to  $x$. Moreover, $\ell(x)$ is continuous and satisfies 
\begin{equation}\label{EA3}
\liminf_{|x|\to\infty} \ell(x)\,> \,0\,.
\end{equation}
\item[(A4)] For some constant $M_1\in (0, M]$ we have the following: for any $\delta\in (0, M_1)$
\begin{equation}\label{EA4A}
\liminf_{\abs{x}\to \infty} \inf_{s\in(0, M_1-\delta)}\frac{f(x, s)}{s^{3-\gamma}}>0, \quad \text{and}\quad \limsup_{\abs{x}\to\infty} \sup_{s\in(M_1+\delta, 2M)}\frac{f(x, s)}{s^{3-\gamma}} < 0\,.
\end{equation}
\item[(A5)] (Strict monotonicity) For any $\kappa_1>\kappa_2>0$ it holds that
$$\inf_{x\in\Rd} \left(\frac{f(x, \kappa_2)}{\kappa^{3-\gamma}_2}-\frac{f(x, \kappa_1)}{\kappa^{3-\gamma}_1}\right)>0.$$
\end{enumerate}
Condition (A1)--(A3) would be used to establish existence of a positive solution whereas (A4)-(A5) would be used to prove uniqueness of positive solution. Conditions (A4) would be
useful to find asymptotic of the positive solutions at infinity. 

We also consider another class of $f$ satisfying:
\begin{itemize}
\item[(B1)] $f:\Rd\times[0, \infty)\to \R$ is continuous and $f(x, 0)=0$. Also, $f(x, \cdot)$ is locally Lipschitz in $(0, \infty)$ uniformly with respect to $x$.
\item[(B2)]  For some $M>0$ we have $f(x, M)\leq 0$ in $\Rd$.
\item[(B3)] For some $\alpha\in (0, 3-\gamma]$,
the limit $\ell(x)\df\lim_{s\to 0} \frac{f(x, s)}{s^\alpha}$ exists and convergence is uniform in $x$. In particular, $\ell(x)$ is continuous. Moreover, $\inf_{\Rd}\ell(x)>0$.
\item[(B4)] For some constant $M_1\in (0, M]$ we have the following: for any $\delta\in (0, M_1)$ we have
\begin{equation}\label{EB4A}
\liminf_{\abs{x}\to \infty} \inf_{s\in(0, M_1-\delta)}\frac{f(x, s)}{s^\alpha}>0, \quad \text{and}\quad \limsup_{\abs{x}\to\infty} \sup_{s\in(M_1+\delta, 2M)}\frac{f(x, s)}{s^\alpha} < 0\,,
\end{equation}
\item[(B5)] (Strict monotonicity) For any $\kappa_1>\kappa_2>0$ it holds that
$$\inf_{x\in\Rd} \left(\frac{f(x, \kappa_2)}{\kappa^{3-\gamma}_2}-\frac{f(x, \kappa_1)}{\kappa^{3-\gamma}_1}\right)>0.$$
\end{itemize}

A typical example of such $f$ would be $f(x, s)=s^\alpha(a(x)-b(x) s^{4-\alpha})$ where $a, b$ are bounded, positive continuous functions and
$\lim_{\abs{x}\to \infty} a(x)=1=\lim_{\abs{x}\to \infty} b(x)$, or $f(x, s)=a(x) s^\alpha(1-s^{4-\alpha})$ for some positive $a$. The readers may observe that
(B3) implies (A3) with $\alpha=3-\gamma$ and (B1) is weaker than (A1).

The nonlinearity  $f$ is often referred  as the Fisher-KPP (for Kolmogorov, Petrovsky and Piskunov) type nonlinearity.  This problem is closely related to the one studied in \cite{BHN,BHR}. The authors in \cite{BHN, BHR} considered the equation 
$$\trace(a(x)D^2 u(x)) + q(x)\cdot\grad u(x) + f(x, u)=0\quad \text{in}\; \Rd,$$
for a Fisher-KPP  type nonlinearity $f$ and established existence and uniqueness of positive solution. One
of the key assumptions imposed on the coefficients is
$$\liminf_{|x|\to\infty}\, (4\alpha(x) f_s(x, 0)-|q(x)|^2)\,>\,0\,,$$
where $\alpha(x)$ denotes the smallest eigenvalue of $a(x)$. This condition plays a key role in the
construction of a suitable subsolution \cite[Lemma~3.1]{BHR}.
 Since we are dealing with a degenerate and nonlinear
operator an analogous condition for the current problem should not be the same as above. However, in that spirit, it is also interesting to ask that if the equation becomes degenerate, it also forces $q$ to vanish. In particular, if $\alpha(x)$ tends to $0$ at infinity,
we have $|q(x)|\to 0$ at infinity. This leads to our hypothesis (Q) above.
 We show in
Proposition~\ref{P4.1} that condition (Q) is sufficient to construct a subsolution suitable for our
purpose.

Our main result of this section is the following
\begin{theorem}\label{T4.1}
Under (A1)--(A5) or (B1)--(B5) there exists a unique bounded, positive solution  to
$$\sL u + f(x, u)=0\quad\quad\textrm{in $\Rd$.}$$ 
\end{theorem}

\begin{remark}
It is not hard to see that (A3) is crucial for the existence of non-trivial non-negative solutions.
For instance, suppose that $f(x, s):\Rd\times\R \to\R$ is such that 
$$s:(0, \infty)\mapsto \frac{f(x, s)}{s^3}\quad \text{is strictly decreasing, for every}\; x,$$
and $\lim_{s\to 0+} \frac{f(x, s)}{s^3} \leq 0$ for every $x$. It then follows that $f(x, s)\leq 0$ for
$s\geq 0$. Thus for any bounded, non-negative solution $u$ of $\infdel u + f(x, u)=0$ must satisfy
$\infdel u\geq 0$ and therefore, by Liouville property \cite{TB01, BD01}, we get $u$ to be constant. This also implies
$f(x, u)=0$ and hence, $u=0$.
\end{remark}

Let us also  mention the non-existence result. The condition \eqref{ET4.2A} below implies
that $\lambda_{\Rd}(\sL+\ell)\geq 0$ and therefore, is consistent with \cite[Proposition~6.1]{BHR}.
Also, condition (Q) is not imposed in the theorem below.

\begin{theorem}\label{T4.2}
Suppose that for some positive continuous function $\Lyap$ with $\inf_{\Rd}\Lyap>0$
 we have
\begin{equation}\label{ET4.2A}
\sL\Lyap + \ell(x)\Lyap^{3-\gamma}\;\leq\; 0\quad \text{in}\; \Rd\,,
\end{equation}
where $\ell(x)=\lim_{s\to 0}\frac{f(x, s)}{s^{3-\gamma}}$, uniformly in $x$.
Furthermore, assume that $(0, \infty)\ni s\mapsto \frac{f(x, s)}{s^{3-\gamma}}$ is strictly decreasing
for every $x\in\Rd$ and for any $0<\delta_1\leq\delta_2$ we have
\begin{equation}\label{ET4.2B}
\limsup_{\abs{x}\to\infty} \sup_{s\in[\delta_1, \delta_2]}\, f(x, s)< 0\,.
\end{equation}
Then any bounded, non-negative solution to $\sL u + f(x, u)=0$ 
must be $0$.
\end{theorem}

\begin{remark}
If $\limsup_{|x|\to\infty}\ell(x)<0$ and $s\mapsto \frac{f(x, s)}{s^{3-\gamma}}$ is decreasing
for every $x\in\Rd$, then we have \eqref{ET4.2B}. In particular, let $\ell(x)\leq-\varepsilon$ for $|x|\geq R_\varepsilon$
for some $\varepsilon, R_\varepsilon>0$. For $0<\delta_1\leq\delta_2$, it then follows  that
\begin{align*}
\sup_{|x|\geq R_\varepsilon}\; \sup_{s\in [\delta_1, \delta_2]} \frac{f(x, s)}{s^{3-\gamma}}
&\leq\, \sup_{|x|\geq R_\varepsilon}\; \left(\sup_{s\in [\delta_1, \delta_2]} \frac{f(x, s)}{s^{3-\gamma}}-\frac{f(x, \delta_1)}{\delta_1^{3-\gamma}}\right)
+ \sup_{|x|\geq R_\varepsilon}\; \frac{f(x, \delta_1)}{\delta_1^{3-\gamma}}
\\
&\leq \, \sup_{|x|\geq R_\varepsilon}\; \frac{f(x, \delta_1)}{\delta_1^{3-\gamma}}
\\
&\leq \, \sup_{|x|\geq R_\varepsilon}\ell(x)\leq -\varepsilon.
\end{align*}
This gives \eqref{ET4.2B}.
\end{remark}
A typical example of $f$ satisfying \eqref{ET4.2B} is $f(x, s)=s^{3-\gamma}(a(x)-b(x)s)$
with $\lim_{|x|\to\infty} a(x)<0$ and $b$ is positive, vanishing at infinity.
The following example gives existence of $\Lyap$ satisfying \eqref{ET4.2A}.
\begin{example}
(a) For $\ell\leq 0$ we may take $\Lyap=1$.\\
(b) Suppose that $q(x)\cdot x\leq -\kappa|x|$ for $|x|\geq\epsilon>0$ and some $\kappa>0$.
Take $\Lyap(x) = e^{\delta \theta(x)}$ where $\theta$ is a non-negative $\cC^2$ function satisfying
$\theta(x)=|x|$ for $|x|\geq \epsilon$. Also, we let $\theta$ to attain its minimum only at $0$ and $\grad\theta(x)\neq 0$ for $x\neq 0$.
Then a direct calculation gives us,
for $x\neq 0$, that
\begin{align*}
\sL \Lyap 
= e^{\delta(3-\gamma) \theta(x)}\left[\delta^{3-\gamma}\abs{\grad\theta}^{-\gamma} \infdel\theta + 
\delta^{4-\gamma} \abs{\grad\theta}^{4-\gamma}+
\delta^{3-\gamma} q(x)\cdot\grad\theta \abs{\grad\theta}^{2-\gamma} \right].
\end{align*}
In $\sB_\epsilon^c$, we have
\begin{align*}
\sL \Lyap 
 & = \Lyap^{3-\gamma}(x)\left[\delta^{4-\gamma} + \delta^{3-\gamma} q(x)\cdot x |x|^{-1}\right]
 \\
 &\leq \Lyap^{3-\gamma}(x)\left[\delta^{4-\gamma} -\kappa \delta^{3-\gamma}\right].
\end{align*}
Fixing $\delta=\kappa/2$ it follows from above that
$$\sL\Lyap\,\leq\, \left(\Theta_\epsilon \Ind_{\sB_\epsilon}(x)-\delta^{4-\gamma}\right)\Lyap^{3-\gamma},$$
where
$$\Theta_\epsilon = \sup_{\sB_\epsilon} \left[\delta^{3-\gamma}\abs{\grad\theta}^{-\gamma} \infdel\theta + \delta^{4-\gamma} \abs{\grad\theta}^{4-\gamma}+
\delta^{3-\gamma} q(x)\cdot\grad\theta \abs{\grad\theta}^{2-\gamma} \right].$$
Thus, if we have $\ell$ satisfying 
$$\ell\leq -\Theta_\epsilon \Ind_{\sB_\epsilon}(x)+\delta^{4-\gamma}$$
in $\Rd$, we have \eqref{ET4.2A}.
\end{example}

The remaining part of this section is devoted to the proofs of Theorem~\ref{T4.1} and ~\ref{T4.2}.
We start by constructing a test function which would play a key role in our analysis.

\begin{proposition}\label{P4.1}
Assume (Q) above.  Then for every $\delta>0$ there exists 
$R_1, R_2>0$ and a smooth function $\psi\gneq 0$ such that for any point $\abs{x_0}\geq R_1$ we have
$$\sL\phi^{x_0} + \delta (\phi^{x_0})^{3-\gamma}> 0\quad \text{in}\; \sB_{R_2}(x_0),\; \phi^{x_0}>0\;\; \text{in}\; \sB_{R_2}(x_0), \quad \text{and}\quad \phi^{x_0}=0\quad \text{on}\; \partial \sB_{R_2}(x_0),$$
where $\phi^{x_0}(\cdot)=\psi(\cdot-x_0)$. In particular, we have 
$$\lambda_{\sB_{R_2}(x_0)}(\sL)\; \leq\; \delta.$$
\end{proposition}

\begin{proof}
Fix $\delta>0$. For some $\varepsilon>0$, we define
$$\psi(x)=\exp(-\frac{1}{1-\abs{\varepsilon x}^2}), \quad \abs{x}<\varepsilon^{-1}.$$
Pick $R_0$ such that $|q(x)|\leq \varepsilon$ for $|x|\geq R_0$. We shall fix a choice of $\varepsilon$ later depending on $\delta$. Let $|z|\geq R_0 + \varepsilon^{-1}$, and
define $\phi^z(x)=\phi(x)=\psi(x-z)$ for $x\in \sB_{\varepsilon^{-1}}(z)$. Then  direct calculations give us
\begin{align*}
&\ginfdel \phi + q(x)\cdot \grad\phi \abs{\grad\phi}^{2-\gamma} + \delta \phi^{3-\gamma}
\\
&\geq \phi^{3-\gamma} \Bigl[\frac{ (2\varepsilon^2 \abs{x-z})^{4-\gamma} }{(1-|\varepsilon(x-z)|^2)^{8-\gamma}} - \frac{(2\varepsilon^2)^{3-\gamma} \abs{x-z}^{2-\gamma} }{(1-|\varepsilon(x-z)|^2)^{6-\gamma}} -\frac{2 (2\varepsilon^2)^{4-\gamma} \abs{x-z}^{4-\gamma} }{(1-|\varepsilon(x-z)|^2)^{7-\gamma}}
\\
&\qquad - \norm{q}_{\sB}\frac{ (2\varepsilon^2 \abs{x-z})^{3-\gamma} }{(1-|\varepsilon(x-z)|^2)^{6-\gamma}} + \delta \Bigr]
\\
&\geq 2^{3-\gamma} \frac{\phi^{3-\gamma}}{(1-|\varepsilon(x-z)|^2)^{8-\gamma}} \Bigl[ 2 (\varepsilon^2 \abs{x-z})^{4-\gamma}  - \varepsilon^{2(3-\gamma)} \abs{x-z}^{2-\gamma} (1-|\varepsilon(x-z)|^2)^2
\\
&\quad - 4 \varepsilon^{2(4-\gamma)} \abs{x-z}^{4-\gamma} (1-|\varepsilon(x-z)|^2)
 - \varepsilon  \varepsilon^{2(3-\gamma)} \abs{x-z}^{3-\gamma} (1-|\varepsilon(x-z)|^2)^2
 \\
 &\qquad + \frac{\delta}{2^{3-\gamma}} (1-|\varepsilon(x-z)|^2)^{8-\gamma} \Bigr].
\end{align*}
Now take $\delta_1\in (0, 1)$ and and consider $\abs{\varepsilon (x-z)}^2\geq (1-\delta_1)$. Since $\abs{\varepsilon (x-z)}\leq 1$ , we note that
$1-\abs{\varepsilon (x-z)}^2\leq \delta_1$ and
\begin{align*}
&2 (\varepsilon^2 \abs{x-z})^{4-\gamma}  - \varepsilon^{2(3-\gamma)} \abs{x-z}^{2-\gamma} (1-|\varepsilon(x-z)|^2)^2
\\
&\quad - 4 \varepsilon^{2(4-\gamma)} \abs{x-z}^{4-\gamma} (1-|\varepsilon(x-z)|^2)
 - \varepsilon  \varepsilon^{2(3-\gamma)} \abs{x-z}^{3-\gamma} (1-|\varepsilon(x-z)|^2)^2
\\
&\qquad \geq 2 (\varepsilon)^{4-\gamma} (1-\delta_1)^{\frac{4-\gamma}{2}} -\varepsilon^{4-\gamma} \delta_1^2 - 4 \varepsilon^{4-\gamma} \delta_1 - \varepsilon^{4-\gamma} \delta_1^2
\\
&\qquad= \varepsilon^{4-\gamma} \left[2(1-\delta_1)^{\frac{4-\gamma}{2}} -  2\delta_1^2 -4\delta_1 \right]>0\,,
\end{align*} 
for $\delta_1$ small, uniformly in $\varepsilon\in (0, 1)$. Thus, for $1-\delta_1\leq \abs{\varepsilon (x-z)}^2\leq 1$, we have
$$\ginfdel \phi + q(x)\cdot \grad\phi \abs{\grad\phi}^2 + \delta \phi^3>0.$$
Now we consider the situation $1-\delta_1> \abs{\varepsilon (x-z)}^2$. Then, we have $(1- \abs{\varepsilon (x-z)}^2)>\delta_1$. Therefore, we obtain from the
above calculation that
\begin{align*}
&\ginfdel \phi + q(x)\cdot \grad\phi \abs{\grad\phi}^2 + \delta \phi^3
\\
&\quad\geq 2^{3-\gamma} \frac{\phi^{3-\gamma}}{(1-|\varepsilon(x-z)|^2)^{8-\gamma}} \left(-8 \varepsilon^{4-\gamma} + \frac{\delta\delta_1^{8-\gamma}}{2^{3-\gamma}}\right)>0\,,
\end{align*}
for $\varepsilon$ small enough. Thus, with this choice of $\varepsilon$, we find that
$$\infdel \phi + q(x)\cdot \grad\phi \abs{\grad\phi}^{2-\gamma} + \delta \phi^{3-\gamma}\,>0\quad \text{in}\; \sB_{\varepsilon^{-1}}(z).$$
We choose $R_1=R_0+\varepsilon^{-1}$ and $R_2=\varepsilon^{-1}$.

By Corollary~\ref{C3.1}, $\lambda_{\sB_{\varepsilon^{-1}}(z)}(\sL)\leq \delta$ for any $z$ satisfying $|z|\geq R_1 + \varepsilon^{-1}$. This
completes the proof.
\end{proof}
We start by proving an existence result.
\begin{lemma}\label{L4.1}
Suppose condition (Q) and  one of the followings hold:
\begin{itemize}
\item[(a)] (A1)--(A3).
\item[(b)] (B1)--(B3).
\end{itemize}
Then there exists a nontrivial non-negative solution of 
$$\sL u(x) + f(x, u)=0 \quad\quad x\in \Rd.$$
\end{lemma}

\begin{proof}
First we consider (b).
Thanks to Theorem~\ref{T2.1} and ~\ref{T2.3}, we can  apply monotone iteration method to find a solution.
Since $f$ need not be Lipschitz all the way upto $0$, we need to modify the proof a bit. Due to
(B3) we can find $\epsilon_0, \delta>0$ satisfying
\begin{equation}\label{EL4.1A}
\inf_{\Rd} f(x, s)\geq 2\delta s^\alpha, \quad \text{for}\; s\leq \epsilon_0\,.
\end{equation}
Now for every $\varepsilon\in (0, \epsilon_0/2)$ we define $f_\varepsilon(x, s)= f(x, \varepsilon +s)$.
Note that $f_\varepsilon$ is locally Lipschitz in $[0, \infty)$. We first find a non-negative
nontrivial solution to
\begin{equation}\label{EL4.1B}
\sL u_\varepsilon + f_\varepsilon(x, u_\varepsilon)=0\quad \text{in}\; \Rd.
\end{equation}
By (B2), the constant function $M-\varepsilon$ is a super-solution  to \eqref{EL4.1B}.
Using Proposition~\ref{P4.1}, we can find a ball $\sB$ and an principal eigenpair 
$(\varphi, \lambda_\sB)$ satisfying 
$$\sL\varphi +\lambda_{\sB}\varphi^{3-\gamma}=0\quad \text{in}\; \sB\,,$$
and $\lambda_\sB\leq \delta$. Also, normalize
$\varphi$ so that $\norm{\varphi}_\infty=\frac{\epsilon_0\wedge 1}{2}$.
Then, using \eqref{EL4.1A},
\begin{align*}
\sL \varphi  + f_\varepsilon(x, \varphi) 
&\geq  f(x, \varepsilon +\varphi)-\lambda_\sB \varphi^{3-\gamma}
\\
&\geq  2 \delta (\varepsilon +\varphi)^\alpha -\delta \varphi^{3-\gamma}
\\
&= 2 \delta (\varepsilon +\varphi)^\alpha -\delta \varphi^{\alpha} >0
\quad \text{in}\; \sB.
\end{align*}
Thus we have a subsolution to \eqref{EL4.1B} in $\sB$. Note that the subsolution vanishes at the 
boundary. Denote by $\bar{u}=M-\varepsilon$. Let $\sigma$ be large enough to satisfy 
$$\sigma > \sup_{x\in\Rd}(\mathrm{Lip}(f_\varepsilon(x, \cdot)) \; \text{on}\; [0, M]).$$
Fix $\sB_n$ large enough to contain $\sB$, and define a sequence of functions $\{u_k\}$ as follows:
$u_1=\bar{u}$, and 
\begin{align*}
\sL u_{k+1}-\sigma u_{k+1} &= f_\varepsilon(x, u_k) -\sigma u_k \quad \text{in}\; \sB_n,
\\
u_{k+1}&=0\quad \text{on}\; \partial\sB_n.
\end{align*}
Existence of solution follows from the arguments of Theorem~\ref{T2.3}. By Theorem~\ref{T2.1}
it also follows that $u_1\geq u_2\geq u_3\geq \ldots\geq 0$. Employing the comparison principle
in $\sB$ we also have $u_k\geq \varphi$ for all $k$. 
Therefore, using Lemma~\ref{L2.1}, we can pass the limit in $\{u_k\}$ to find a solution to 
$$\sL u_{n, \varepsilon} + f_\varepsilon(x, u_{n, \varepsilon})=0\quad \text{in}\; \sB_n(0),$$
with $\varphi\leq u_{n, \varepsilon}\leq M$ in $\sB_n(0)$. Also, by Lemma~\ref{L2.1} , we note that the $u_{n, \varepsilon}$ is locally Lipschitz uniformly in $n, \varepsilon$. Thus we can extract 
a subsequence converging to some $u_\varepsilon\in \cC(\Rd)$ solving
$$\sL u_\varepsilon + f_\varepsilon(x, u_\varepsilon)=0\,,$$
in $\Rd$ and $\varphi\leq u_\varepsilon \leq M$ in $\Rd$. This gives \eqref{EL4.1B}.
We again use a similar argument to pass the limit to $\varepsilon\to 0$, and obtain a solution
$$\sL u  + f(x, u)=0, $$
in $\Rd$ and $\varphi\leq u \leq M$ in $\Rd$.

Now we consider (a). In this case the proof is more straight-forward since $f$ is locally Lipschitz
in $[0, \infty)$. We just need to find a positive subsolution in a ball $\sB$. Note that by \eqref{EA3}
there exists $\delta>0$ such that 
$$\ell(x)>2\delta,$$
for all $|x|\geq R$, for some $R$. Again, $f(x, s)\geq (\ell(x)-\delta)s^{3-\gamma}$ for all $x$ and $s\leq \epsilon_0$. Then applying Proposition~\ref{P4.1} we can find a ball $\sB\Subset\sB_R^c$, and
and eigenfunction $\varphi$ with $\norm{\varphi}_\infty\in (0, \epsilon_0/2)$ satisfying
$$\sL \varphi  + f(x, \varphi)\geq 0\quad \text{in}\; \sB,$$
giving a positive subsolution in $\sB$. Hence, we can repeat the arguments as above to find a non-trivial, non-negative solution.
\end{proof}

The following result shows a strong maximum principle.
\begin{lemma}\label{L4.2}
Suppose that either (A3) or (B3) holds.
If $v$ is a non-negative super-solution to $\sL v + f(x, v)= 0$, then either we have $v\equiv 0$ or $\inf_{\Rd} v>0$. 
\end{lemma}

\begin{proof}
For the first part, we show that either $v\equiv 0$ or $v>0$. Consider $D=\{x\in\Rd\; :\; v(x)=0\}$. Since $v$ is continuous, by Lemma~\ref{L2.1},
we must have $D$ closed. We show that $D$ is also open. Take $z\in D$. 
Using (A3) above we can find a ball $\sB(z, r)$ such that $c(x)\df\frac{f(x, v(x))}{v^{3-\gamma}(x)}$
bounded. Thus $v$ is a super-solution  of
$$\sL v- \norm{c}_{L^\infty(\sB(z, r))} v^{3-\gamma}=0\quad \text{in}\; \sB(z, r).$$
Applying Theorem~\ref{T2.2} we obtain $v=0$ in $\sB(z, r)$. Thus $D$ is open. 
Now consider (B3). Since $f(x, s)\geq 0$ for all $s$ small, we can choose $\sB(z, r)$ small enough so
that $f(x, v)\geq 0$ in $\sB(z, r)$. Hence, $\sL v\leq 0$ in $\sB(z, r)$ implying $v=0$ in $\sB(z, r)$,
by Theorem~\ref{T2.2}. Hence $D$ is open. Therefore, either $D=\emptyset$ or
$D=\Rd$. This proves the first part.

Next we suppose that $v>0$ in $\Rd$. We give a proof with the assumption (B3) and the proof 
assuming (A3) would be analogous.
The idea of the proof is to use the subsolution constructed in Proposition~\ref{P4.1}. Fix $\delta>0$ small enough so that
$$\ell(x)\geq 3\delta\quad \text{for all}\; |x|\geq R,$$
for some $R>0$. 
By our assumption of $f$, there exists $\epsilon_0>0$ satisfying $f(x, s)\geq 2\delta s^{\alpha}\geq 2\delta s^{3-\gamma}$ for all $|x|\geq R$ and $s\in [0, \epsilon_0)$.
Choose $R_1 (\geq R), R_2$ and $\psi$ from Proposition~\ref{P4.1} with the above choice of $\delta$. Normalize $\psi$ so that $\norm{\psi}_\infty=\kappa\leq \epsilon_0$. Here
we choose $\kappa$ small enough so that 
$$\kappa<\inf_{\sB_{R_1+ 2R_2}} v.$$
We show that
\begin{equation}\label{EL4.2A}
\inf_{\Rd} v\,\geq \kappa.
\end{equation}
From Proposition~\ref{P4.1}, we note that for any $|z|\geq R_1+ 2R_2$, we have for $\phi(x)=\phi^z(x)=\psi(x-z)$ that
\begin{equation}\label{EL4.2B}
\sL\phi + f(x, \phi)\geq -\delta\phi^{3-\gamma} + 2\delta \phi^{3-\gamma}=\delta \phi^{3-\gamma}\quad \text{in}\; \sB_{R_2}(z).
\end{equation}
Pick $z\in\Rd$ with $|z|\geq R_1+R_2$ and let
$\gamma:[0, 1]\to \Rd$ be the line joining $0$ to $z$. Define
$$t^*=\sup\{t\in[0, 1]: \psi(\cdot - \gamma(t))< v\quad \text{in}\; \sB_{R_2}(\gamma(t))\}.$$
Clearly, $t^*>0$ due to continuity. We need to show that $t^*=1$. Suppose
that $t^*<1$. Then in the ball
$\widehat\sB=\sB_{R_2}(\gamma(t^*))$ we have $\phi(\cdot)=\psi(\cdot - \gamma(t^*))\leq v$ and it
must touch $v$ at some point in $\widehat\sB$.
 By our choice it also evident that $|\gamma(t^*)|\geq R_1+R_2$.
Also, $\phi$ satisfies \eqref{EL4.2B} in $\widehat\sB$ and vanishes on the boundary of $\widehat\sB$.
As in the proof of Theorem~\ref{T2.1}, we consider
$$w_\varepsilon(x, y)= \phi(x)-v(y) -\frac{1}{4\varepsilon}\abs{x-y}^4, \quad x, y\in{\widehat\sB}.$$
Clearly, $\max w_\varepsilon> 0$. Let $(x_\varepsilon, y_\varepsilon)$ be a maximizer. As shown in Theorem~\ref{T2.1}, we may
also assume that $x_\varepsilon, y_\varepsilon\to z\in \widehat\sB$ as $\varepsilon\to 0$, since the maximum of $(\phi-v)$ can not be attained on the boundary.
 Hence repeating the arguments of Theorem~\ref{T2.1} we arrive at 
\begin{align*}
\delta \phi^{3-\gamma}(x_\varepsilon)-f(x_\varepsilon, \phi(x_\varepsilon))\leq -f(y_\varepsilon, v(y_\varepsilon)) 
+ \omega(\abs{x_\varepsilon-y_\varepsilon})(1+ \left(\varepsilon^{-1}|x_\varepsilon-y_\varepsilon|^3\right)^{3-\gamma}).
\end{align*}
Letting $\varepsilon\to 0$ and using  \eqref{ET2.1B},
 we obtain $\delta \phi^{3-\gamma}(z)\leq 0$, contradicting the fact $\phi$ is positive inside $\sB_1$. This proves \eqref{EL4.2A}.
\end{proof}

\begin{remark}\label{Re-eigen}
As far as the existence of a bounded positive solution is concerned, the condition (A3) can be relaxed. For instance, a condition weaker than \eqref{EA3} is
$$\lim_{n\to\infty} \lambda_{\sB_n}(\sL+\ell)\,<\, 0\,.$$
Under this hypothesis we can construct a positive subsolution $\underline{u}_k$ of $\sL u + f(x, u)=0$ in an arbitrary large ball $\sB_k$ with a Dirichlet condition on the boundary. By scaling we can also keep this
subsolution smaller that $M$. Then the arguments of Lemma~\ref{L4.1} shows that the solution obtained by monotone iteration should stay above $\underline{u}_k$ for all $k$. Thus, the solution has to be
positive in $\Rd$.
\end{remark}

Combining Lemma~\ref{L4.1} and Lemma~\ref{L4.2} we obtain the existence of a positive solution. Now we proceed for the uniqueness. In some cases, we can obtain
the uniqueness as a consequence of the Liouville property. For instance, if we consider $f(x, s)=s^{3-\gamma}(1-s)$
and $q$ is compactly supported, then from the Liouville property (Theorem~\ref{T2.5}) it follows that 
there is no non-constant solution of $\sL u + f(x,u)=0$ in $\Rd$ if $u\leq 1$. But we cannot apply Liouville theorem in our general setting. Also, the method of
\cite{BHR} fails to apply for degenerate operator, as we are dealing with a degenerate nonlinear operator.
To establish the uniqueness we first find the asymptotic of  solutions at infinity.

\begin{lemma}\label{L4.3}
Suppose that either (A4) or (B4) holds. Then for any positive super-solution $v$ of 
$\sL v + f(x, v)= 0$ in $\Rd$ we have $\liminf_{\abs{x}\to\infty} v(x)\geq M_1$.
\end{lemma}

\begin{proof}
 We may assume, without loss of generality, that $M_1=1$. We only provide a proof under the hypothesis (B4).
Fix $\varepsilon\in (0, 1)$. Let $\kappa>0$ be small enough to satisfy
$$4\kappa<\liminf_{\abs{x}\to\infty} \inf_{s\in (0, 1-\varepsilon)} \frac{f(x, s)}{s^\alpha}.$$
 Thus, there exists $r_\circ>0$ such that
\begin{equation}\label{EL4.3A}
\inf_{s\in (0, 1-\varepsilon)} \frac{f(x, s)}{s^\alpha} > 3\kappa\quad \text{for all}\; \abs{x}\geq r_\circ.
\end{equation}
Pick $r_1, r_2$ and $\psi$ from Proposition~\ref{P4.1} for $\delta=\kappa$.  
Normalize $\norm{\psi}_\infty=1$ and define $\phi^z_\varepsilon(x)=(1-\varepsilon) \psi(x-z)$.
We claim that 
\begin{equation}\label{EL4.3B}
\phi^z_\varepsilon(\cdot) \leq v(\cdot)\quad \text{in}\; \sB_{r_2}(z), \quad \text{for all }\; |z|\; \text{large}.
\end{equation}
If not, there would exist $|z|> 2(r_1+r_2+r_\circ)$ such that $\phi^z_\varepsilon(x_0)> v(x_0)$ for some $x_0\in \sB_{r_2}(z)$.
Define
$$\eta = \max\{t>0\; :\; t\phi^z_\varepsilon\,<\, v\quad \text{in}\; \sB_{r_2}(z)\}.$$
It is easily seen that $\eta\in (0, 1)$ and furthermore, $\eta\phi^z_\varepsilon$ should touch $v$ from below inside $\sB_r(z)$ as $v>0$ and $\phi^z_\varepsilon$ vanishes on the boundary of
$\sB_{r_2}(z)$. 
Again, $\norm{\eta\phi^z_\varepsilon}_\infty< (1-\varepsilon)$ and, by Proposition~\ref{P4.1}, we have
\begin{align*}
\sL (\eta\phi^z_\varepsilon) + f(x, \eta\phi^z_\varepsilon)&\geq
-\kappa\, \eta^{3-\gamma} (\phi^z_\varepsilon)^{3-\gamma} + f(x, \eta\phi^z_\varepsilon) 
\\
&\geq  f(x, \eta\phi^z_\varepsilon) - \kappa\,  (\eta\phi^z_\varepsilon)^{\alpha}
\\
&\geq (\eta\phi^z_\varepsilon)^\alpha \left(\frac{f(x,  \eta\phi^z_\varepsilon)}
{(\eta\phi^z_\varepsilon)^\alpha} -\kappa \right)
\\
&\geq 2\kappa (\eta\phi^z_\varepsilon)^\alpha,
\end{align*}
by \eqref{EL4.3A}.
Then, repeating the argument of Lemma~\ref{L4.2} (or Theorem~\ref{T2.1}) we get a contradiction. This proves the claim \eqref{EL4.3B}. Since the maximum of $\psi$ is $1$, it follows from \eqref{EL4.3B}  that
$$\liminf_{\abs{x}\to\infty} v(x)\geq 1-\varepsilon.$$
The arbitrariness of $\varepsilon$ implies the result.
\end{proof}

Let us now prove an upper bound on the asymptotic at infinity.

\begin{lemma}\label{L4.4}
Suppose that either (A4)--(A5) or (B4)--(B5) hold.
Let $u$ be a bounded, positive subsolution to $\sL u + f(x, u)= 0$  in $\Rd$. Then we have $\sup_{\Rd} u\leq M$. Furthermore, we also have $\limsup_{\abs{x}\to\infty} u(x)\leq M_1$.
\end{lemma}

\begin{proof}
We only provide a proof under the hypothesis (A4)--(A5).
On the contrary, we assume that $\sup u=M_\circ>M$. We fix $\varepsilon>0$ such that $u(x_0)> M_\circ-\varepsilon>M+\varepsilon$ for some $x_0$. For 
simplicity we may assume that $x_0=0$. Note that
\begin{equation}\label{EL4.4A}
\sup_{x\in\Rd}\; \sup_{s\in[M+\varepsilon, M_\circ+2]} f(x, s)<0.
\end{equation}
Indeed, 
\begin{align*}
& \sup_{x\in\Rd}\; \sup_{s\in[M+\varepsilon, M_\circ+2]} \frac{f(x, s)}{s^{3-\gamma}}
\\
&\;\leq \sup_{x\in\Rd}\; \sup_{s\in[M+\varepsilon, M_\circ+2]} \left(\frac{f(x, s)}{s^{3-\gamma}}-\frac{f(x, M+\varepsilon)}{(M+\varepsilon)^{3-\gamma}}\right)
+ \sup_{x\in\Rd} \left(\frac{f(x, M+\varepsilon)}{(M+\varepsilon)^{3-\gamma}}-\frac{f(x, M)}{M^{3-\gamma}}\right) 
\\
&\qquad + \sup_{\Rd}\frac{f(x, M)}{M^{3-\gamma}}
\\
&\leq \sup_{x\in\Rd} \left(\frac{f(x, M+\varepsilon)}{(M+\varepsilon)^{3-\gamma}}-\frac{f(x, M)}{M^{3-\gamma}}\right)<0,
\end{align*}
by (A5). Define $\theta(x)=\abs{x}^2-1$ and $\theta_r(x)=\theta(\frac{1}{r}x)$. Then for $r>0$ large it is easily seen that 
$$\sup_{x\in\Rd}\; \sup_{s\in[M+\varepsilon, M_\circ+2]} f(x, s)< - (\ginfdel\theta_r 
+ \norm{q}_{L^\infty} \abs{\grad\theta_r}^{3-\gamma}), \quad \text{in}\; \sB_r(0).$$
Note that $\theta_r(0)=-1$. 
Let 
$$\beta=\inf\{ \kappa\in [M+\varepsilon, M_\circ+2]\; :\; \kappa+\psi_r > u \quad \text{in}\;  \sB_r(0)\}.$$
Since $u(0)>M_\circ-\varepsilon$, it follows that $\beta > M_\circ + 1-\varepsilon$ as $M_\circ+ 1-\varepsilon - \theta_r(0)=M_\circ-\varepsilon$. Again, $\beta+\theta_r$ should touch 
$u$ from above inside $\sB_r(0)$ since $(\beta+\theta_r)> M_\circ+1-\epsilon$ on $\partial \sB_r(0)$. We call $v=\beta+\theta_r$. Then
\begin{align*}
\sL v + f(x, v)&\leq \sup_{x\in\Rd}\; \sup_{s\in[M+\varepsilon, M_\circ+2]} f(x, s) + \ginfdel v + \norm{q}_{L^\infty} \abs{\grad v}^{3-\gamma}=-\delta<0.
\end{align*}
Thus $v$ is super-solution touching $u$ from above. We can now follow the arguments of Lemma~\ref{L4.2} (or Theorem~\ref{T2.1}) to obtain that 
$\delta<0$ which is a contradiction. This proves the first part.

Now we come to the second part and the proof is quite similar to above. For simplicity assume $M_1=1$. Suppose that $\tilde{M}_\circ\df\limsup_{\abs{x}\to\infty} u(x)> 1$. Then we can
find $\varepsilon\in (0 , 1)$ so that $u(x)>\tilde{M}_\circ-\varepsilon> 1+ \varepsilon$ for infinitely many $x$ tending to infinity. On the other hand, by \eqref{EA4A}, we have
$$\sup_{|x|\geq r_\circ}\; \sup_{s\in[1+\varepsilon, 2M]} f(x, s)<0,$$
for some $r_\circ>0$.
Therefore, we can apply the argument as above by suitably translating the test function $v$ and then get a contradiction. Hence we must have $\limsup_{\abs{x}\to\infty} u(x)\leq1$. This completes the proof.
\end{proof}
Finally, we establish the uniqueness.
\begin{lemma}\label{L4.5}
Suppose that either (A4)--(A5) or (B4)--(B5) hold. Then there exists a unique, bounded positive solution to 
$\sL u + f(x, u)=0$ in $\Rd$.
\end{lemma}

\begin{proof}
In view of Lemma~\ref{L2.1} we note that any bounded solution has to be globally Lipschitz.
Let $w_1, w_2$ be two solutions to $\sL u + f(x, u)=0$ in $\Rd$.
In view of Lemma~\ref{L4.3} and Lemma~\ref{L4.4} we see that $\lim_{\abs{x}\to \infty} w_1(x)=\lim_{\abs{x}\to \infty} w_2(x)=M_1$. 
Suppose that there exists $x_0\in\Rd$ satisfying $w_1(x_0)>w_2(x_0)$. Define
$$\kappa^* =\max\{t>0\; :\; t w_1< w_2\quad \text{in}\; \Rd\}.$$
Since $\inf w_2>0$, it follows that $\kappa^*>0$. Also, $\kappa^*<1$. Thus, $\liminf_{\abs{x}\to\infty} (w_2(x)-\kappa^* w_1(x))>0$. It then implies that $w_2-\kappa^* w_1$ must vanish somewhere in 
$\Rd$ i.e. $\min_{\Rd}(w_2(x)-\kappa^* w_1(x))=0$. 

As before, we consider the coupling function
$$w_\varepsilon(x, y)= \kappa^* w_1(x)-w_2(y) -\frac{1}{2\varepsilon}\abs{x-y}^4, \quad x, y\in\Rd.$$
Note that there will be a pair of point $(x_\varepsilon, y_\varepsilon)$ attending maximum of $w_\varepsilon$. Pick a $\delta\in (0, 1-\kappa^*)$ small and a number $K$ large enough 
so that
$$\kappa^* w_1(x)\leq \kappa^*+\delta, \quad w_2(x)\geq \kappa^*+2\delta\quad \text{for all}\; |x|\geq K.$$
Thus, for $\abs{x-y}\leq 1$ and $|y|\geq K+1$ we have $w_\varepsilon(x, y)<-\delta$.  Again, for $\abs{x-y}\geq 1$, $w_\varepsilon(x, y)<0$ for all $\varepsilon$ small.
Since $w_\varepsilon(x_\varepsilon, y_\varepsilon)\geq 0$, it follows that 
 $\abs{x_\varepsilon} + \abs{y_\varepsilon}\leq K+1$ for all $\varepsilon$. As in the proof of Theorem~\ref{T2.1}, we will also have
$$\lim_{\varepsilon\to 0}\frac{1}{2\varepsilon}\abs{x_\varepsilon-y_\varepsilon}^4=0,\quad |x_\varepsilon-y_\varepsilon|^3=\order(\varepsilon), \quad \text{and}\quad x_\varepsilon, y_\varepsilon\to z.$$
Also, $w_2(z)=\kappa^* w_1(z)>0$. Also,  we have
$$\sL(\kappa^* w_1) + (\kappa^*)^{3-\gamma} f(x, w_1)
\geq (\kappa^*)^{3-\gamma} \left(\sL w_1(x) +  f(x, w_1(x))\right)\geq 0.$$
Thus, arguing as in Theorem~\ref{T2.1}, we obtain
$$-(\kappa^*)^{3-\gamma} f(x_\varepsilon, w_1(x_\varepsilon))\leq - f(y_\varepsilon, w_2(y_\varepsilon))+
\omega(|x_\varepsilon-y_\varepsilon|) (\varepsilon^{-1}\abs{x_\varepsilon-y_\varepsilon}^3)^{3-\gamma}.$$
Letting $\varepsilon\to 0$, and arguing similar to Theorem~\ref{T2.1}, we find
\begin{align*}
0 &\geq f(z, \kappa^* w_1(z))-(\kappa^*)^{3-\gamma} f(z, w_1(z))
\\
&\geq (\kappa^* w_1(z))^{3-\gamma} \inf_{y\in\Rd} \left(\frac{f(y, \kappa^* w_1(z))}{(\kappa^* w_1(z))^{3-\gamma}}- \frac{f(y, w_1(z))}{(w_1(z))^{3-\gamma}}\right)>0,
\end{align*}
by (A5). This is a contradiction and therefore, $w_1\leq w_2$. Similarly, we have $w_2\leq w_1$. Hence, we complete the proof.
\end{proof}

\begin{proof}[Proof of Theorem~\ref{T4.1}]
The existence follows from Lemma~\ref{L4.1} and Lemma~\ref{L4.2} whereas the uniqueness follows from Lemma~\ref{L4.5}.
\end{proof}

Finally, we prove Theorem~\ref{T4.2}.
\begin{proof}[Proof of Theorem~\ref{T4.2}]
Since we have $s\mapsto \frac{f(x, s)}{s^{3-\gamma}}$ strictly decreasing, it is
easily seen that $f(x, s)\leq \ell(x) s^{3-\gamma}$ for all $s\geq 0$.
 Thus it follows from \eqref{ET4.2A} that
\begin{equation}\label{ET4.2C}
 \sL\Lyap + f(x, \Lyap(x))\;\leq\; 0\quad \text{in}\; \Rd\,.
\end{equation}
Suppose that there exists $u\gneq 0$, bounded, satisfying
\begin{equation}\label{ET4.2D}
\sL u + f(s, u)=0 \quad \text{in}\; \Rd.
\end{equation}
Then, first part of the proof of Lemma~\ref{L4.2} implies that $u>0$ in $\Rd$.
For any $\kappa<1$, define $u_\kappa=\kappa u$. Using monotonicity
 and \eqref{ET4.2D} it then follows that
\begin{equation}\label{ET4.2E}
\sL u_\kappa + f(s, u_\kappa)\,\geq\,0 \quad \text{in}\; \Rd.
\end{equation}
Choose $\epsilon_0\in (0, 1)$ so that $f(x, s)\leq \ell(x)s^{3-\gamma}$ for
all $s\in[0, \epsilon_0)$. Now we claim that for any $\kappa<\frac{\epsilon_0}{\norm{u}_\infty+1}$,
$\norm{u}_\infty\df\norm{u}_{L^\infty(\Rd)}$, we have
\begin{equation}\label{ET4.2F}
u_\kappa(x)\,\leq\, \Lyap(x) \quad \text{for all}\; x\in\Rd\,.
\end{equation}
To prove the claim first we observe from the proof of Lemma~\ref{L4.4} and \eqref{ET4.2B} that 
$\lim_{\abs{x}\to\infty} u_\kappa(x)=0$. Let
$$\beta=\sup\{t\geq 0\; :\;  t\,u_\kappa< \Lyap\}.$$
Since $\Lyap>0$, it is obvious that $\beta>0$. To prove \eqref{ET4.2F} we need show that $\beta\geq 1$.
We assume by contradiction that $\beta<1$. Since $\liminf_{\abs{x}\to\infty}(\Lyap(x)-\beta u_\kappa(x))>0$, $\beta u_\kappa$ must touch $\Lyap$ from below in $\Rd$. Consider the
coupling function 
$$w_\varepsilon(x, y)= \beta u_\kappa(x)-\Lyap(y) -\frac{1}{2\varepsilon}\abs{x-y}^4, \quad x, y\in\Rd
$$
as in Lemma~\ref{L4.5}, and then following the arguments of Lemma~\ref{L4.5} we find
$$\lim_{\varepsilon\to 0}\frac{1}{2\varepsilon}\abs{x_\varepsilon-y_\varepsilon}^4=0,\quad |x_\varepsilon-y_\varepsilon|^3=\order(\varepsilon), \quad \text{and}\quad x_\varepsilon, y_\varepsilon\to z,$$
and $\Lyap(z)=\beta u_\kappa(z)\in (0, \epsilon_0)$. Also, 
$\ell(z) \Lyap^{3-\gamma}(z)\geq f(z, \Lyap(z))$. Then, repeating the arguments of
Lemma~\ref{L4.5} and using \eqref{ET4.2E} we arrive at a contradiction. This proves $\beta\geq 1$, giving us \eqref{ET4.2F}.

Now observe that \eqref{ET4.2A} (and therefore, \eqref{ET4.2C}) holds if we replace $\Lyap$ by $\mu\Lyap$ for any $\mu>0$. Thus, 
we obtain from \eqref{ET4.2F} that $\kappa u \leq \mu \Lyap $ for any $\mu>0$ and 
$\kappa<\frac{\epsilon_0}{\norm{u}_\infty+1}$. But this is not possible since $u>0$ in $\Rd$. This gives us
a contradiction. Hence $u\equiv 0$.
\end{proof}

\subsection*{Acknowledgement}
The research of Anup Biswas was supported in part by DST-SERB grants EMR/2016/004810 and MTR/2018/000028.

\end{document}